\documentclass[a4paper,onecolumn,superscriptaddress,10pt]{compositionalityarticle}
\pdfoutput=1
\usepackage[utf8]{inputenc}
\usepackage[english]{babel}
\usepackage[T1]{fontenc}
\usepackage{amsmath}
\usepackage{hyperref}

\usepackage{tikz}
\usepackage{lipsum}

\usepackage{amssymb,amscd,amsthm,extarrows,verbatim,amsmath,color,fancyhdr,mathrsfs}
\usepackage{algorithm}
\usepackage{algpseudocode}
\usepackage{hyperref}
\usepackage{multirow}
\usepackage{graphicx}
\usepackage{hyperref}
\usepackage{turnstile}
\usepackage{arydshln}
\usepackage{enumitem}
\usepackage{appendix}

\usepackage{tikz-cd}
\usetikzlibrary{arrows}

\usepackage{xcolor}
\usepackage{pagecolor}
\newcommand{\compositionalityIssue}{2}
\newcommand{\compositionalityVolume}{8}
\newcommand{\paperdoinumber}{10.46298/compositionality-\compositionalityVolume-\compositionalityIssue}
\newcommand{\paperdoi}{https://doi.org/\paperdoinumber}
\newcommand{\receiveddate}{2023-04-26}
\newcommand{\accepteddate}{2026-02-25}
\newcommand{\R}{\mathbf{R}}
\newcommand{\N}{\mathbf{N}}
\newcommand{\OO}{\mathcal{O}}
\newcommand{\Z}{\mathbb{Z}}

\newcommand{\mcf}{\mathcal{F}}
\def\arrvline{\hfil\kern\arraycolsep\vline\kern-\arraycolsep\hfilneg}

\usepackage[all,2cell]{xy}
\UseTwocells
\newcommand{\Digra}{\mathbf{Digra}}
\newcommand{\SDigra}{\mathbf{SDigra}}
\newcommand{\Gra}{\mathbf{Gra}}
\newcommand{\Part}{\mathbf{Part}}
\newcommand{\WGra}{\mathbf{WGra}}
\newcommand{\ood}{\mathbf{OOD}}
\DeclareMathOperator{\card}{card}
\DeclareMathOperator{\ob}{ob}

\newtheorem{theorem}{Theorem}[section]
\newtheorem{prop}[theorem]{Proposition}
\newtheorem{lem}[theorem]{Lemma}
\newtheorem{cor}[theorem]{Corollary}
\newtheorem{defn}[theorem]{Definition}
\newtheorem{rem}[theorem]{Remark}
\newtheorem{ex}[theorem]{Example}

\begin{document}
%%%%%%%%%%%%%format

\title{Rank-based linkage I: triplet comparisons and oriented simplicial complexes}
\date{}
\author{R.W.R. Darling}
\email{rwdarli@uwe.nsa.gov}
\homepage{https://www.nsa.gov}
\orcid{0000-0002-7296-9633}
\author{Will Grilliette}
\email{wbgrill@nsa.gov}
\orcid{0000-0003-0310-780X}
\affiliation{National Security Agency, Fort George G Meade MD, MD 20755-6844, USA}
\author{Adam Logan}
\email{adam.m.logan@gmail.com}
\orcid{0000-0001-6157-3521}
\affiliation{Tutte Institute for Mathematics and Computing, P.O. Box 9703, Terminal, Ottawa, ON, Canada}
\affiliation{School of Mathematics and Statistics, 4302 Herzberg Laboratories, Carleton University,
1125 Colonel By Drive, Ottawa, ON, K1S 5B6, Canada}

\maketitle
%%%%%%%%%%%%%%%%%%%%%%%%%%%%%%%%%%%%%%%%%%%%%%%%%%%%%%%%%%%%%%%%%%%%%%%%%%%%%%%%%%%%%%
\begin{abstract}
Rank-based linkage is a new tool for summarizing a collection $S$ of
objects according to their relationships. These objects are not mapped
to vectors, and ``similarity'' between objects need be neither numerical nor symmetrical.
All an object needs to do is rank nearby objects by similarity to itself, using
a \texttt{Comparator} which is transitive, but need not be consistent with any metric on the whole set.
Call this a ranking system on $S$.
Rank-based linkage is applied to the $K$-nearest neighbor digraph
derived from a ranking system. Computations occur on a
2-dimensional abstract oriented simplicial complex whose 
faces are among the points, edges, and triangles of the line graph of
the undirected $K$-nearest neighbor graph on $S$.
In $|S| K^2$ steps it builds an edge-weighted linkage graph $(S, \mathcal{L}, \sigma)$
where $\sigma(\{x, y\})$ is called the in-sway between objects $x$ and $y$.
Take $\mathcal{L}_t$ to be the links whose in-sway is at least $t$, and
partition $S$ into components of the graph $(S, \mathcal{L}_t)$, for varying $t$.
Rank-based linkage is a functor from a category of ``out-ordered'' digraphs to a category
of partitioned sets, with the practical consequence that augmenting the set of objects
in a rank-respectful way gives a fresh clustering which does not ``rip apart'' the previous one.
The same holds for single linkage clustering in the metric space context, but not for typical
optimization-based methods. 
Orientation sheaves play in a fundamental role and ensure that partially overlapping data sets can be ``glued'' together.
Open combinatorial problems are presented in the last section.
\end{abstract}

{\small
\noindent \textbf{Keywords:}
clustering, non-parametric statistics, functor, ordinal data, sheaf \\
\textbf{MSC class: } Primary: 62H30; Secondary: 05C20, 05E45, 05C76 
}

\newpage
\setcounter{tocdepth}{1}
\tableofcontents

%%%%%%%%%%%%%%%%%%%%%%%%%%%%%%%%%%%%%
\section{Introduction}

\subsection{Data science motivation}
The rank-based linkage algorithm is motivated by needs of graph data science:
\begin{enumerate}
    \item[(a)] Summarization of {\em property graphs}, where nodes and edges have non-numerical attributes, but a node
    can compare two other nodes according to similarity to itself.
    \item[(b)] {\em Graph databases} require partitioning algorithms whose output remains stable when a batch of new nodes and edges is added.
    \item[(c)] Big data partitioning algorithms which are {\em robust} against data dropout, inhomogenous scaling, and effects of high degree vertices.
\end{enumerate}

\subsection{Topological data analysis paradigm}
Topology is concerned with properties of a geometric object that are preserved under continuous deformations. 
Rank-based linkage is
a new topological data analysis technique applicable, for example, to large
edge-weighted directed graphs. Its results are
invariant under monotone increasing transformations of
weights on out-edges from each vertex.
Equivalence classes of such graphs form a category 
under a suitable definition of
morphism  (Section \ref{s:catoodigr}) and rank-based linkage is a functor into
the category of partitioned sets (of the vertex set)
(Section \ref{s:partset}), in a technical sense that satisfies (b) above.
Single linkage clustering 
exhibits a similar functorial property
in the case of finite metric spaces
\cite{car}; however, such functoriality is rare among unsupervised learning
techniques.

\subsection{Triplet comparison paradigm for data science}

\begin{figure}[t]
\caption{\textbf{Ranking system on four objects: }\textit{
The labels $1, 2, 3$ on the out-arrows from object $v_1$
indicate that $v_4, v_2, v_3$  are ranked in that order 
by similarity to $v_1$. }
}
\label{f:k4}
\begin{center}
\begin{tikzpicture}[scale=2]
\draw (0,0) node (1) {$v_1$};
\draw (0,1) node (4) {$v_4$};
\draw (-30:1) node (3) {$v_3$};
\draw (210:1) node (2) {$v_2$};
\draw[->,red] (1) to [bend left=15] node[left] {1} (4);
\draw[<-,blue] (1) to [bend right=15] node[right] {2} (4);
\draw[->,red] (1) to [bend left=15] node[below] {2} (2);
\draw[<-,orange] (1) to [bend right=15] node[above] {3} (2);
\draw[->,red] (1) to [bend left=15] node[above] {3} (3);
\draw[<-] (1) to [bend right=15] node[below] {3} (3);
\draw[->,blue] (4) to [bend left=15] node[left] {1} (3);
\draw[<-] (4) to [bend left=30] node[right] {1} (3);
\draw[->,blue] (4) to [bend right=15] node[right] {3} (2);
\draw[<-,orange] (4) to [bend right=30] node[left] {2} (2);
\draw[->,orange] (2) to [bend right=15] node[above] {1} (3);
\draw[<-] (2) to [bend right=30] node[below] {2} (3);
\end{tikzpicture}
%\scalebox{0.4}{\includegraphics{IMAGES/ArcsK4TaggedWithRanks.png} }
\end{center}
\end{figure}

Consider exploratory data analysis on a set $S$ of $n$ objects, where
relations are restricted to those of the form:
\begin{center}
    \textit{object $y$ is more similar to object $x$ than object $z$ is.}
\end{center}
Kleindessner \& Von Luxburg \cite{kle} call such a relation a \textit{triplet comparison}.
Readers may be familiar with triplet comparison systems arising from
a metric $d(\cdot, \cdot)$ on $S$, where the relation
means $d(x,y) < d(x,z)$. In sequels to this paper, we will argue
that a broader class of triplet comparison systems (such as those arising from Markov transition probabilities) is natural in 
data science applications.

\subsection{Presentations of ranking systems}\label{s:crs}
There are several equivalent ways to present triplet comparison systems, each
useful in its own context. These include:
\begin{enumerate}
    \item [(a)] Comparator $\prec_x$ places a 
    total order on $S\setminus\{x\}$ for each object $x$, where
    $y \prec_x z$ whenever $r_x(y) < r_x(z)$; \texttt{Comparator} is an interface in the Java language \cite{java}.
    
    \item [(b)] For each $x \in S$, a
    bijection $r_x(\cdot): S \setminus \{x\} \to \{1, 2, \ldots, n-1\}$ 
    assigns ranks to objects other than $x$.
 Collect these bijections into an $n \times n$ \textbf{ranking table} $T$, whose
 $(i, j)$ entry $r_i(j)$ is the rank awarded to object $x_j$ by the
ranking based at object $x_i$, taking $r_i(i)=0$.
The ranking table (e.g. Table \ref{t:3conc10points}) is an efficient data storage device.
    
    \item [(c)] Orient (i.e. assign directions to the edges of) the line graph of the complete graph on the set $S$, as in Figure \ref{f:linek4}, and elaborated in Section \ref{s:orientations}. 
    The simplicial complex of points, edges, and triangles of the line graph provides
    a setting for our main results in Sections \ref{s:asc}, \ref{s:rblformal}, and \ref{s:functor}.
    
\end{enumerate}

\begin{figure}
\caption{\textbf{Orientation of the line graph: }
\textit{
The labelled digraph of Figure \ref{f:k4} is presented in terms of its octahedral line graph.
The arrow $\{v_1, v_4\} \rightarrow \{v_1, v_3\}$ corresponds to 
the fact that $v_4$ is ranked as closer to $v_1$ than $v_3$ is to $v_1$. 
The absence of cycles shows that the ranking system is concordant, i.e
consistent with a total order on the nodes with the source $\{v_1, v_3\}$ 
at the bottom and the sink $\{v_3, v_4\}$ at the top.
Nodes
of the line graph are 0-simplices, edges are
1-simplices, and triangular faces are 2-simplices in Section \ref{s:asc}.}
 [Figure courtesy of Marcus Bishop]
}
\label{f:linek4}
\begin{center}
\begin{tikzpicture}[scale=1]
\draw (0,4) node (12) {$\left\{v_1,v_2\right\}$};
\draw (6,1) node (13) {$\left\{v_1,v_3\right\}$};
\draw (2,-1) node (14) {$\left\{v_1,v_4\right\}$};
\draw (-2,1) node (23) {$\left\{v_2,v_3\right\}$};
\draw (-6,-1) node (24) {$\left\{v_2,v_4\right\}$};
\draw (0,-4) node (34) {$\left\{v_3,v_4\right\}$};
\draw[->] (12) -- (13);
\draw[<-] (12) -- (14);
\draw[<-] (12) -- (23);
\draw[<-] (12) -- (24);
\draw[<-] (13) -- (14);
\draw[<-] (13) -- (23);
\draw[<-] (13) -- (34);
\draw[->] (14) -- (24);
\draw[<-] (14) -- (34);
\draw[->] (23) -- (24);
\draw[<-] (23) -- (34);
\draw[<-] (24) -- (34);
\end{tikzpicture}
\end{center}
\end{figure}
% Claude revision

\subsection{Orientations of the line graph} \label{s:orientations}
The reader is advised to study Figures \ref{f:k4} and \ref{f:linek4} carefully.
The former shows the complete graph $\mathcal{K}_S$, where the label $k$ on a directed edge
$x \rightarrow y$ means that $y$ is ranked $k$ by $x$ among the $n-1$ elements of $S \setminus \{x\}$.
Figure \ref{f:linek4} shows the corresponding line graph $L(\mathcal{K}_S)$, whose vertices (``0-simplices'')
are unordered pairs of objects, and whose edges (``1-simplices'') are pairs such as $\{\{x, y\}, \{x, z\}\}$
with an element in common. Interpret a triplet
comparison $y \prec_x z)$ as an assignment of a direction to a specific edge of
the line graph:
\(
    \{x,y\} \rightarrow \{x,z\}
\). We often abbreviate this to 
\(
    xy \rightarrow xz
\).
To emphasize: {\em the ranks which appear as labels on directed edges of the complete graph 
(Figure \ref{f:k4}) provide directions
of the arcs of the line graph (Figure \ref{f:linek4}).}
The combinatorial input to data science algorithms 
will be a simplicial complex, consisting of a subset of the vertices, edges, and triangular faces 
(``2-simplices'') 
of the line graph $L(\mathcal{K}_S)$, 
after orientations are applied to the 1-simplices. Full details are found in Section \ref{s:asc}.

Our use of simplicial complexes is distinct conceptually
from other uses in network science, such as \cite{her}. \textbf{Warning: } the objects of $S$
are {\em not} the vertices of the simplicial complex.

\subsection{Enforcer triangles and voter triangles} \label{s:enfvot}
A careful look at Figure \ref{f:linek4} shows that there are two kinds of triangular faces:
 the first kind, called 
\textit{enforcers}, take the form
\begin{equation}\label{e:enforcer3}
    \begin{tikzcd}{} & v_j v_m \arrow{dr}{} \\v_i v_m \arrow{ur}{} \arrow{rr}{} && v_k v_m\end{tikzcd}
\end{equation}
where the arrows declare that, under the comparator at $v_m$, the integers in the rank table satisfy $r_m(i) < r_m(j) < r_m(k)$. Enforcer 
 triangles are enforcing transitivity of the total order induced by $v_m$'s ranking. In an enforcer triangle, four distinct objects in $S$ are mentioned.

 There remain the second kind of triangles, called \textit{voters}, where for distinct $i, j, k$,
 \begin{equation}\label{e:voter3}
     \begin{tikzcd}{} & v_i v_j \arrow{dr}{} \\v_i v_k \arrow{ur}{} \arrow{rr}{} && v_j v_k\end{tikzcd} 
 \end{equation}
The arrows declare that the integers in the rank table satisfy
$r_i(k) < r_i(j)$, $r_j(i) < r_j(k)$, and $r_k(i) < r_k(j)$.
A voter triangle mentions only three objects, each of which (here $v_i$, $v_j$, and $v_k$) ``votes'' on
a preference between between the other two. See Figure \ref{f:pushout} for more examples.

\subsection{Three kinds of ranking system}\label{s:3kinds}

\begin{enumerate}
    \item 
A \textbf{ranking system} (without ties) corresponds to an arbitrary ranking table, and to
an orientation of $L(\mathcal{K}_S)$ such that all the enforcer triangles are acyclic.
There are no restrictions on the voter triangles.

\item
A \textbf{concordant ranking system} is an acyclic orientation of $L(\mathcal{K}_S)$.
Baron and Darling \cite{bar} proved that concordant ranking systems are precisely those arising from metrics on $S$.  Figure \ref{f:linek4} shows an example.

\item
Intermediate between (1) and (2) are the \textbf{3-concordant ranking systems},
in which there are no directed 3-cycles. In other words, both enforcer triangles and voter triangles are
acyclic; however, cycles are allowed among closed walks of length four or greater.
It appears that 3-concordant ranking systems are especially appropriate
for rank-based algorithms based on triplet comparisons, and include ranking systems 
not arising from metrics on $S$. Table \ref{t:3conc10points} and Figure \ref{f:counterex} show examples.
\end{enumerate}

\begin{table}[]
    \centering
    \begin{tabular}{c|c|c|c|c|c|c|c|c|c|c|}
 object & 0 & 1 & 2 & 3 & 4 & 5 & 6 & 7 & 8 & 9 \\ \hline
$r_0(\cdot)$ & 0 & 7 & 6 & 4 & 8 & 3 & 1 & 5 & 9 & 2\\ 
$r_1(\cdot)$ &  5 & 0 & 3 & 6 & 7 & 8 & 1 & 4 & 9 & 2\\ 
$r_2(\cdot)$ &  7 & 6 & 0 & 1 & 9 & 5 & 3 & 4 & 8 & 2\\
$r_3(\cdot)$ &  2 & 9 & 5 & 0 & 4 & 6 & 3 & 8 & 7 & 1\\
$r_4(\cdot)$ &  9 & 8 & 7 & 4 & 0 & 3 & 6 & 2 & 1 & 5\\
$r_5(\cdot)$ &  4 & 9 & 3 & 6 & 5 & 0 & 1 & 7 & 8 & 2\\
$r_6(\cdot)$ &  1 & 6 & 5 & 4 & 8 & 3 & 0 & 7 & 9 & 2\\
$r_7(\cdot)$ &  7 & 9 & 1 & 8 & 2 & 3 & 5 & 0 & 4 & 6\\
$r_8(\cdot)$ &  9 & 6 & 7 & 5 & 1 & 3 & 8 & 2 & 0 & 4\\
$r_9(\cdot)$ &  7 & 5 & 4 & 2 & 8 & 3 & 1 & 6 & 9 & 0 \\
\hline
    \end{tabular}
    \caption{\textit{Ranking table for a uniformly random 3-concordant ranking system on ten objects, 
    obtained by sampling 9.5 billion times from all ranking systems (i.e.
    a random permutation of nine objects for each row) until a 3-concordant example
    was obtained. The $(i, j)$ entry $r_i(j)$ is the rank awarded to object $j$ by the
ranking based at object $i$, taking $r_i(i)=0$.
    For example, $r_1(6) = 1$ means that object 1 ranks object 6 as closer to itself than
    any of the objects $\{0,\dots,9\} \setminus \{1,6\}$.
    There are several 4-cycles. Puzzle: find one.}
    }
    \label{t:3conc10points}
\end{table}

\subsection{Natural examples of non-metric ranking systems arise in data science} \label{s:nonmetrixex}
In Section \ref{s:ewdigraph} edge-weighted digraphs will be seen as the prototypical examples
of non-metric ranking systems; see Section \ref{s:migration}. Meanwhile, here are some examples of more specific comparators.

\begin{itemize}

\item 
\textit{Lexicographic comparators in property graphs: } Each object in $S$ has multiple categorical fields.
For example a {\em packet capture} interface may specify fields such as source IP, destination IP,
port, number of packets sent, and time stamp. List these in some order of importance.
Given $x$, decide whether $y$ or $z$ is more similar to $x$ by passing through the
fields in turn, comparing the entries for $y$ and $z$ in that field with the entry for $x$.
The first time $y$ agrees with $x$ while $z$ differs, or vice versa, the tie is broken in
favor of the one which agrees with $x$. The \texttt{Comparator.thenComparing()} method 
in the Java language \cite{java} implements this.

\item \textit{Information theory: } 
$S$ consists of finite probability measures, and $q \prec_p r$ if  
Kullback-Leibler divergences satisfy $D(\mathbf{p} \| \mathbf{q}) <
D(\mathbf{p} \| \mathbf{r})$, where
\(
D(\mathbf{p} \| \mathbf{q}):= \sum_i p_i \log{ (p_i / q_i)}.
\)
This is generally not 3-concordant: see \cite{knnd}.

\item 
$S$ is the state space of a \textit{Markov chain }with transition probability 
$P$, and $y \prec_x z$ if $P_{x,y} > P_{x, z}$.
Random walk on a weighted (hyper)graph gives a concordant ranking system,
as will be explained in a future work.

\item  $S$ is a collection of \textit{senders of messages} on a network during a specific time period. 
Declare that $y \prec_x z$ if $x$ sends more messages to $y$ than to $z$,
or if $x$ sends the same number to each but the last one to $y$ was later than the last one to $z$.
In contrast to undirected graph analysis, messages received by $x$ play no part in its Comparator.
\end{itemize}

Why do we favor 3-concordant ranking systems over concordant ones?
For algorithms such as rank-based linkage (Algorithm \ref{a:rbl}) and 
partitioned local depth \cite{ber} which study triangles
in which a particular 0-simplex (i.e. object pair) is a source,
the extra generality is welcome, and carries no cost because
cycles longer than three do not matter.
On the other hand a 3-cycle appears as a ``hole'' in the data, and is something
we prefer to avoid.

\subsection{Goal and outline of present paper}
The influential paper of Carlsson and Memoli \cite{car} has illuminated
the defects of many clustering schemes applied to finite metric spaces.
In particular these authors postulate that a clustering algorithm should be a functor from a
category of finite metric spaces to a category of partitioned sets (Section \ref{s:functormotivation}).

Section \ref{s:functor} supplies a categorical framework which offers
a functorial perspective on clustering schemes for ranking systems.
This perspective suggests a specific hierarchical clustering scheme called rank-based linkage.
Our main result is Theorem \ref{th:summary}, verifying that 
rank-based linkage is a functor between suitable categories.

As for the rest of the paper:
\begin{enumerate}
    \item[(a)] Users of our algorithm will find all the practical details they need in Section 
\ref{s:linkage}.
    \item[(b)] Section \ref{s:asc} on orientation sheaves on simplicial complexes supplies the structural rationale
for the design of the algorithm.
    \item[(c)] Section \ref{s:rblformal} proves combinatorial properties which establish algorithm correctness.
    \item[(d)] Section \ref{s:functor} derives functorial properties, with full category theory precision.
    \item[(e)] Section \ref{s:gluers} describes how the algorithm behaves when two experimenters merge their data --- an important but neglected aspect of data science.
    \item[(f)] Section \ref{s:sampling} deals with random sampling of 3-concordant ranking systems;
 Section \ref{s:open} in particular draws attention to open problems in sampling and counting these systems.
\end{enumerate}

\subsection{Related literature}
Clustering based on triplet comparisons has been been studied by
Perrot et al.~\cite{per}, who attempt to fit a numerical  ``pairwise similarity matrix''
based on triplet comparisons.
Mandal et al.~\cite{man} consider hierarchical clustering and measure the goodness of dendrograms.
Their goals are different to ours, which can be seen as a scalable redesign of \cite{ber}.
%%%%%%%%%%%%%%%%%%%%%%%%%%%%%%%%%%%%%%%%%%%%%%%%%%%%%%%%%%%%
\section{Hierarchical clustering based on linkage}\label{s:linkage}

\subsection{Need for $K$-nearest-neighbor approximation}
In this section we deal with practical graph computation on 
billion-scale edge sets. The number of voter triangles for
a ranking system on $n$ objects is $O(n^3)$, which is impractical
at this scale; indeed, cubic complexity limits the application
of the related PaLD algorithm \cite{ber}. 
To dodge this burden, we work with a $K$-nearest-neighbor approximation\footnote{
In suitable settings \cite{knnd}, this could be
achieved with $O(n K^2 \log_K{n})$ work by $K$-nearest neighbor descent. 
} 
to the ranking system. This lowers the computational effort from $O(n^3)$ to $O(n K^2)$,
and memory from $O(n^2)$ to $O(n K)$. 
The integer $K \leq n$ should be large enough so that the $K$ nearest neighbors of any object
provide sufficient data for clustering decisions,
but small enough so that linking decisions are ``local''
(these are informal, not mathematical, definitions).

\subsection{Out-ordered digraphs}\label{s:outorddigr}
In Section \ref{s:3kinds} we considered a ranking system on
a set $S$ of $n$ objects where every $x \in S$ applies a total order to $S \setminus \{x\}$.
Relax this total order to a preorder (i.e. a transitive, reflexive relation):
for each $x \in S$ produce a set $\Gamma(x) \subset S \setminus \{x\}$ of no more than $K$ 
nearest neighbors\footnote{
When $S$ is the vertex set of a sparse graph, $\Gamma(x)$ is typically a subset of the graph neighbors of $x$. 
If a vertex $x$ has degree less than $K$, then $|\Gamma(x)| < K$. 
} of $x$
with three properties:
\begin{enumerate}
    \item[(i)] the elements of $\Gamma(x)$ (``friends'' of $x$) are totally ordered by $\prec_x$;
    \item[(ii)] each element of $\Gamma(x)$ precedes every element (``non-friend'') of $S \setminus (\Gamma(x) \cup \{x\})$;
    \item[(iii)] no order relation exists between any pair of elements of $S \setminus (\Gamma(x) \cup \{x\})$.
\end{enumerate}
\begin{defn}\label{d:outorddigr}
A set $S$ together with friend sets $\{\Gamma(x), x \in S\}$ and pre-orders which satisfy (i), (ii), and (iii) above
is called an out-ordered digraph on $S$, abbreviated to $(S,\Gamma,\preceq)$.
\end{defn}

Results later in the paper do not depend on whether the sets $\{\Gamma(x), x \in S\}$ have identical
cardinality, but the upper bound $K$ is used in complexity estimates. 
Call $x$ and $z$ ``mutual friends'' if $z \in \Gamma(x)$ and $x \in \Gamma(z)$. 
Whereas a full ranking of $S \setminus \{x\}$ according to $\prec_x$  for every $x$
would take $O(n^2 \log{n})$ work in total, a ranking of every $\Gamma(x)$ needs only $O(n K \log{K})$ work.
Section \ref{s:catoodigr} will revisit Definition \ref{d:outorddigr} in terms of category theory.

\subsection{Concrete presentation as edge-weighted digraph} \label{s:ewdigraph}
Consider edge-weighted directed graphs on $S$,
whose arcs consist of:
\[
A_{\Gamma}:=\{x \rightarrow y, \, x \in S, \, y \in \Gamma(x)\}.
\]
The weight on the arc $x \rightarrow y$ is denoted $w_x(y) \in \R$, and
the weight function is denoted $w_.(\cdot)$. 
Interpret $w_x(y)$ as a scale-free measure of similarity of $y$ to $x$,
as perceived by $x$.
We further require that, for each $x$, the
values $\{w_x(y),  y \in \Gamma(x)\}$ are distinct real numbers.
Introduce an equivalence relation on such edge-weighted digraphs 
as follows. 

\begin{defn}\label{d:rankequivw}
Edge-weighted digraphs $(S, A_{\Gamma}, w_.(\cdot))$ and $(S', A_{\Gamma}', w'_.(\cdot))$ are rank-equivalent 
if $S = S'$, $\Gamma(x)= \Gamma'(x)$ for all $x \in S$, and if for each $x \in S$,
the total ordering of $\Gamma(x)$ according to $y \mapsto w_x(y)$
and the total ordering according to $y \mapsto w'_x(y)$ coincide.
\end{defn}

%An illustration is given in Figure \ref{f:rankequivw}. \adam{I'm not really sure this is necessary.} 
We omit the proof of the following assertion.

\begin{lem}\label{l:rankequivw}
An out-ordered digraph on $S$ is an equivalence class of rank-equivalent digraphs.
Weighted digraph $(S, A_{\Gamma}, w_.(\cdot))$ belongs to this equivalence class if for all $x \in S$,
\begin{equation}\label{e:rankequivweights}
    y \prec_x y' \Leftrightarrow w_x(y) > w_x(y'), \quad y, y' \in \Gamma(x).
\end{equation}
\end{lem}

\begin{defn}\label{d:riwdg}
Such a representative $(S, A_{\Gamma}, w_.(\cdot))$
of the equivalence class of an out-ordered digraph on $S$, satisfying (\ref{e:rankequivweights}),
is called a representative out-ordered digraph.
\end{defn}
\textbf{Canonical example: } Take $w_x(y) = -k$ if $y$ is ranked $k$ among the $|\Gamma(x)|$ friends of $x$.
\begin{comment}
\begin{figure}
\caption{\textbf{Pair of rank-equivalent digraphs: }\textit{
Ordering of neighbors 1 to 6 by the weight of the arc from vertex 1 is the same
for both stars.
See Definition \ref{d:rankequivw}.}
}
\label{f:rankequivw}
\begin{center}
\scalebox{0.25}{\includegraphics{IMAGES/TwoStarGraphs-rank-equivalent-weight.png} }
\end{center}
\end{figure}
\end{comment}

\textbf{Practical significance: }Lemma \ref{l:rankequivw} allows us to perform rank-based linkage
computations using any directed edge-weighted graph in the correct equivalence class of 
out-ordered digraph on $S$,
provided the weights of arcs are used only for comparison, and not for arithmetic.
Contrast this with graph partitioning based on the combinatorial Laplacian \cite{moh}: in the
latter case, numerical values of edge weights form vectors used for eigenvector
computations, which are sensitive to nonlinear monotonic transformations.

\subsection{Line graph from an out-ordered digraph on $S$ }
There is a ``forgetful view'' of the out-ordered digraph,
 depending only on the friend sets $\{\Gamma(x), x \in S\}$. By abuse of notation, we shall use
 the familiar term ``$K$-NN'' although some friend sets may have fewer than $K$ elements.

\begin{defn}\label{d:unng}
The undirected $\Gamma$-neighbor ($K$-NN) graph $G_{\Gamma}:=(S, \mathcal{U}_{\Gamma})$ consists of all the arcs in 
the out-ordered digraph, stripped of their direction and weight:  
\[
\mathcal{U}_{\Gamma}:= \{\{x, y\}\subset S \mid y \in \Gamma(x) \quad \mbox{or} \quad x \in \Gamma(y) \}
\]
\end{defn}

\begin{defn}\label{d:lingraph}
The line graph $L(G_{\Gamma})$ has vertex set  $ \mathcal{U}_{\Gamma}$ (0-simplices);
the edges (1-simplices) are pairs (such as $\{ \{x, y\}, \{x, z\}\} \subset \mathcal{U}_{\Gamma}$) with an object (here $x$) in common.
\end{defn}

\begin{figure}
\caption{\textbf{Ranking system where vertex $x$ has no friends }\textit{
[Figure courtesy of the referee]}}\label{f:nofriendsofx}
\begin{center}
\begin{tikzpicture}[scale=2]
\draw (0,0) node (1) {$x$};
\draw (0,1) node (4) {$y$};
\draw (-30:1) node (3) {$w$};
\draw (210:1) node (2) {$z$};
\draw[<-,blue] (1) to [bend right=15] node[right] {2} (4);
\draw[<-,orange] (1) to [bend right=15] node[above] {1} (2);
\draw[<-] (1) to [bend right=15] node[below] {3} (3);
\draw[->,blue] (4) to [bend left=15] node[left] {1} (3);
\draw[<-] (4) to [bend left=30] node[right] {1} (3);
\draw[->,blue] (4) to [bend right=15] node[right] {3} (2);
\draw[<-,orange] (4) to [bend right=30] node[left] {2} (2);
\draw[->,orange] (2) to [bend right=15] node[above] {3} (3);
\draw[<-] (2) to [bend right=30] node[below] {2} (3);
\end{tikzpicture}
\end{center}
\end{figure}
There is a weakness in Definition \ref{d:lingraph}. Consider
Figure \ref{f:nofriendsofx}, which differs from the complete directed graph of Figure \ref{f:k4} in that none of $\{y, z, w\}$
is an out-neighbor of $x$. 
In the line graph $L(G_{\Gamma})$ derived from Figure \ref{f:nofriendsofx}, $\{x, y\}$ and  $\{x, z\}$
 are both in $\mathcal{U}_{\Gamma}$ and are regarded as adjacent (they share $x$ in common), for we
have $x \prec_y z$ and $x \prec_z y$. Thus $x \in \Gamma(y)$ and  $x \in \Gamma(z)$, but  $y \notin \Gamma(x)$ and  $z \notin \Gamma(x)$.
An edge such as $\{xy, xz\}$ exists because it was created by other objects’ lists ($y$ and $z$),
but at $x$ there is no local information to order/compare $y$ and $z$. In other words, it seems that no natural
orientation is determined at $x$.
Definition \ref{d:basedlg} is designed to rectify this weakness.

\begin{defn}\label{d:basedlg}
Given an out-ordered digraph $(S,\Gamma,\preceq)$, the based line graph $H_{\Gamma}$ 
is the spanning subgraph pf $L(G_{\Gamma})$ obtained by restricting edges to be only those pairs $\{xy, xz\}$ 
 such that  $\{y, z\} \cap \Gamma(x) \neq \emptyset$. Such an edge is \textbf{based} at $x$.
 In other words 
 \[
V(H_{\Gamma}):= \mathcal{U}_{\Gamma}; \quad \{xy, xz\} \in E(H_{\Gamma}) \Leftrightarrow \{y, z\} \cap \Gamma(x) \neq \emptyset.
 \]
\end{defn}
\begin{rem}
Compare Figure \ref{f:nofriendsofx}: the 1-simplex $\{xy, xz\} \notin E(H_{\Gamma})$ 
since $y \notin \Gamma(x)$ and  $z \notin \Gamma(x)$, so this 1-simplex will not acquire an orientation. 
\end{rem}

\begin{defn}\label{d:orientbasedlg}
Edges (1-simplices) of the based line graph $H_{\Gamma}$ have a canonical orientation:  
\begin{equation}\label{e:1-simplices}
    xy \rightarrow xz \Leftrightarrow \begin{cases}
        (\{y, z\} \subset \Gamma(x) ) \land ( y \prec_x z ) & \mbox{or} \\
        (y \in \Gamma(x) )\land( z \notin \Gamma(x)) &
    \end{cases}
\end{equation}
\end{defn}

\begin{rem}
One option for the orientation occurs when $|\{y, z\} \cap \Gamma(x)| = 1$, and 
comes from the principle that {\em a non-friend of $x$ is less similar to $x$ than is a friend of $x$.}
\end{rem}

\begin{rem}
A computational advantage of using $H_{\Gamma}$ versus the line graph from a complete ranking system
is that the number of orientable 1-simplices is at most $n K (K-1)/2$ elements, instead of $n^3$.
\end{rem}

Definition \ref{d:basedlg} is sufficient preparation for the rank-based linkage algorithm, which we present next. 
The based line graph concept will recur later:
\begin{enumerate}
    \item[(i)] In Sections \ref{s:complex} an \Ref{s:sheafview}, the sheaf of 3-concordant 
    orientations is well-defined because each boundary edge's orientation comes from its unique base.
    \item[(ii)] Categorical treatment of out-ordered digraphs in Lemma \ref{l:alternatef} requires it also.
\end{enumerate}

\subsection{Rank-based linkage algorithm: overview}
If we were dealing with a full ranking system on $S$, not just comparators on friend sets, then
the hierarchical clustering algorithm in this section can be over-simplified as
\begin{center}
    \textit{Link a pair of objects if the corresponding 0-simplex of the line graph is the source\\in at least $t$ voter triangles, for suitable $t$. Then partition according to graph components.}
\end{center}
When we deal with an out-ordered digraph on $S$ (Definition \ref{d:outorddigr}), 
more subtlety is needed; we restrict the 0-simplices, and Definition \ref{d:strictkpert} will determine which voter triangles qualify for inclusion.

\subsection{Data input: edge-weighted digraph}\label{s:datainput}
In our Java implementation \cite{rbl}, the input to the rank-based linkage algorithm
is a flat file, possibly with billions of lines, 
presenting a representative of an out-ordered digraph (Definition \ref{d:riwdg}).
Each line takes the form $(x, y, w_x(y))$, regarded as a directed
edge from object $x$ to object $y \in \Gamma(x)$, with weight $w_x(y)$ of some type
which is \texttt{Comparable}\footnote{
The Java \cite{java} interface \texttt{Comparable} 
is implemented by any type on which there is total ordering, such as integers and floating point numbers.
}.

Let us emphasize that only the relative values of $\{w_x(y), y \in \Gamma(x)\}$ are important. For example, $w_x(y)$ can be taken as minus the rank of $y$ in the ordering of $\Gamma(x)$
given by $\prec_x$. 
%\adam{Maybe this should be mentioned earlier.  There is something to be said for having a canonical representative of every equivalence class in an equivalence relation.}

The cut-down to $K$ nearest neighbors may or may not have been done in advance. In either case,
our Java implementation \cite{rbl} reads the file and stores it in RAM in the data structure of a 
directed edge-weighted graph, with out-degree bounded above by $K$.
Assume henceforward that every object $x \in S$ has degree at least two in the graph $(S, \mathcal{U}_{\Gamma})$.
This condition is enforced \cite{rbl} by 
restricting all subsequent computation to the (undirected) two-core of 
an input digraph, before taking the $\Gamma$-neighbors of each vertex.

\subsection{Canonical rank-based linkage algorithm} \label{s:rbl-steps}

Here are the steps of the rank-based linkage algorithm, based on the graphs of Definitions \ref{d:riwdg} 
and \ref{d:unng}, which are concrete expressions of Definition \ref{d:outorddigr}. 
Its rationale emerges from the functorial viewpoint of Section \ref{s:functor}.

The algorithm will create a new undirected edge-weighted graph 
$(S, \mathcal{L}, \sigma_K)$ called the \textbf{linkage graph} 
(Definition \ref{d:linkage}) whose links $\mathcal{L}$ are the pairs of mutual friends, and whose weight function $\sigma_K: \mathcal{L} \to \Z_+$ is called the \textbf{in-sway}.
We shall describe this first in text, and then in pseudocode as Algorithm \ref{a:rbl}.

\begin{enumerate}
\item[(a)] Determine the set of pairs of mutual friends: these are the links $\mathcal{L}$.

\item[(b)]  Iterate over the links $\mathcal{L}$, possibly in parallel.
    
\item[(c)]  For each pair $\{x, z\} \in \mathcal{L}$ of mutual friends, iterate over $y$ among the common neighbors of $x$ and
of $z$ in the undirected graph $(S, \mathcal{U}_{\Gamma})$. Count those $y$ for which three conditions hold:
\begin{itemize}
    \item $\{x, z\} \cap \Gamma(y)$ is non-empty.
    \item  $y$ is distinct from $x$ and $z$.
    \item $xz$ is the source in the voter triangle $\{xz, xy, yz\}$. The two ways
 $xz$ can be the source\footnote{
All the 1-simplices have a well-defined direction, as we shall see in Definition \ref{d:strictkpert}.
This triangle must have a source and a sink if the out-ordered digraph is 3-concordant.   
    } are:
    \begin{equation}\label{e:source}
  \begin{tikzcd}{} &  \arrow{dl}{} x z \arrow{dr}{} \\x y  \arrow{rr}{} &&  y z\end{tikzcd} 
  \quad \mbox {or} \quad
  \begin{tikzcd}{} &   \arrow{dl}{} x z \arrow{dr}{} \\x y   && \arrow{ll}{}  y z\end{tikzcd}         
    \end{equation}

    \end{itemize}
    
    \item[(d)] The count obtained in (c) is the in-sway $\sigma_K(\{x,  z\})$.     
    
    \item[(e)] At the conclusion of the iteration over $\mathcal{L}$,
    the linkage graph $(S, \mathcal{L}, \sigma_K)$ is finished.
    It supplies a monotone decreasing sequence of undirected subgraphs 
    \((S, \mathcal{L}_t)_{t = 0, 1, 2, \ldots}\)
    where $\mathcal{L}_t$ is defined to be the set of links whose in-sway is at least $t$.
    
    \item[(f)] 
    The hierarchical clustering scheme (Definition \ref{d:rbl}) 
    consist of partitions into graph components of $(S, \mathcal{L}_t)_{t > 0}$. Decreasing $t$ adds more 
    links\footnote{Parameters $K$ and $t$ are analogous to parameters of the distance-based method DBSCAN \cite{est}.},
    and may collapse components together.

    \item[(g)] If  $\max_t |\mathcal{L}_t| \geq n$, the \textbf{critical in-sway} $t_c$ is the largest value of $t$ for which
  $|\mathcal{L}_t| \geq n$. The sub-critical clustering 
    consists of the graph components of $(S, \mathcal{L}_{t_c + 1})$; it is recommended as a starting point
    for examining the collection of partitions.
   
\end{enumerate}

%%%%%%%%%%%%%%%%%%%%%%%%%%%%%%%%%%%%%%%%%%%%
\begin{algorithm}
  \caption{Linkage graph from out-ordered digraph}\label{a:rbl}
  \begin{algorithmic}[1]
\Procedure{Mutual friend list}{$(S, A_{\Gamma})$}
      \State $\mathcal{L} := \emptyset$
      \Comment{Empty set of unordered pairs of objects}
      \For{$x \in S$}
        \For{$z \in \Gamma(x)$}
        \If{$x \in \Gamma(z)$}
        \State $\mathcal{L} := \mathcal{L} \cup \{ \{x, z\} \}$
        \Comment{Include pair in set of mutual friends}
        \EndIf
         \EndFor
    \EndFor
     \State \textbf{return: set of mutual friends $\mathcal{L}$} 
\EndProcedure
%%%%%%%%%%%%%%%%
    \Procedure{In-sway}{$(S, A_{\Gamma}, w_{\cdot}(\cdot))$, $(S,  \mathcal{U}_{\Gamma})$}
    \Comment{Invoke \textsc{Mutual friend list} first}
\For{$ \{x, z\} \in \mathcal{L}$}
\Comment{Iterate through pairs of mutual friends}
\State Set $\sigma_K(\{x,  z\}):=0$
\Comment{In-sway initialization}
\For{$y$ such that $\{x, y\} \in \mathcal{U}_{\Gamma}$ and $\{z, y\} \in \mathcal{U}_{\Gamma}$}
    \Comment{See Definition \ref{d:unng}}
\If{$\{x, z\} \cap \Gamma(y) \neq \emptyset$ and $y \notin \{x, z\}$}
\Comment{Reason: see Definition \ref{d:strictkpert}}
              \If{\textsc{FirstElementIsSource}($\{xz, yz, xy\}$)}   
                      \Comment{Predicate below}
                        \State Add 1 to $\sigma_K(\{x,  z\})$
                    \EndIf                  
                 \EndIf                        
            \EndFor
        \EndFor
      \State \textbf{return: } Linkage graph $(S, \mathcal{L}, \sigma_K)$\\
    \Comment{Correctness proved in Proposition \ref{p:rblcorrectness}}
        \EndProcedure
              \Comment{$\mathcal{L}_t:=\{x, z\}, \,\sigma_K(\{x,  z\}) \geq t\}$}
    \Procedure{FirstElementIsSource}{$\{xz, yz, xy\}$}
    \Comment{See Lemma \ref{l:ksway}}
    \If{$(y \notin \Gamma(x) \quad \mbox{or} \quad z \prec_x y) \quad$ and 
$\quad (y \notin \Gamma(z) \quad \mbox{or} \quad x \prec_z y)$ 
    }
    \State \textbf{return} True
    \Else
    \State \textbf{return} False
    \EndIf
    \EndProcedure
  \end{algorithmic}
\end{algorithm}

\subsection{Design and typical applications of the algorithm}
\subsubsection{Does the ranking system need to be 3-concordant?}
The in-sway computation (\ref{e:source}) requires a decision about the source of each voter
triangle. If there were a 3-cycle in the line graph, the corresponding voter triangle would be
``missing data''. Thus 3-concordance is preferred, but a few exceptions 
(such as may occur in the Kullback-Leibler divergence example of Section \ref{s:nonmetrixex}) merely
induce some bias, and may be tolerated.
\subsection{Does every object need to have exactly $K$ friends?}
If the friends $\Gamma(y)$ are a subset of the neighbors of $y$ in a sparse graph, then
it will sometimes happen that $|\Gamma(y)| < K$. This makes such a $y$ ``less influential" than an object with $K$ neighbors, in
that it participates in fewer voter triangles; this does not upset the algorithm much. On the other hand
if $K = 10$ but we allow one specific $y$ with $|\Gamma(y)| = 50$, then this $y$ will be ``excessively influential'',
in the sense that in (c) it will participate in many more voter triangles than a typical object does.
This would introduce a more serious bias, which we prefer to avoid.
\subsection{Typical input format}
Our initial input takes a list of object pairs with weight as input: this is interpreted
as edge data for an edge-weighted graph which may be
either directed or undirected, according to a run time parameter.
\subsection{Compute effort}
Table \ref{tab:runtimes} indicates that
setting up graph data structures in RAM takes more time than running Algorithm \ref{a:rbl},
which is an $O(n K^2)$ counting operation, executable in multi-core architectures using parallel streams.

\subsection{Computational examples}
\subsubsection{Ten point set}

\begin{figure}
\caption{\textbf{Rank-based linkage on Table \ref{t:3conc10points}:}
\textit{Critical-in-sway is five. Links, labelled by in-sway, are solid for in-sway greater than five. There are two singletons,
and clusters of sizes three and five in the sub-critical clustering.}
}
\label{f:rbl10pts}
\begin{center}

\scalebox{0.4}{\includegraphics{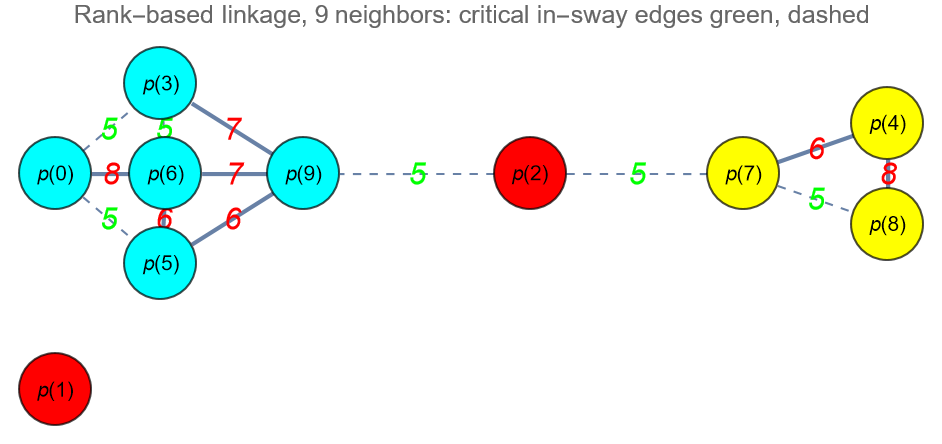} }
\end{center}
\end{figure}

Figure \ref{f:rbl10pts} shows rank based linkage on the ten objects 
whose ranking system appears in Table \ref{t:3conc10points},
with $\Gamma(x):=S \setminus \{x\}$ for each $x$. Observe that the
pairs $\{p(0), p(6)\}$ and $\{p(4), p(8)\}$ both
exhibit the maximum possible value of in-sway, which is 8.
Verify in Table \ref{t:3conc10points} that each of these
0-simplices is the source in all 8 of the voter triangles
in which it occurs. The sub-critical clustering gives
4 clusters, namely two singletons,
and clusters of sizes three and five, respectively.

\begin{figure}
\caption{
\textbf{Migration flows:} \textit{Objects are 199 countries in the 2-core.}
\textit{WHENCE graph: country $x$ compares countries $y$ and $z$ according to how many immigrants it receives from each.}
\textit{WHITHER graph: country $x$ compares countries $y$ and $z$ according to how many
citizens of $x$ have migrated to $y$ or $z$.
Both linkage graphs are restricted to links of in-sway at least 25. 
Besides singletons, WHENCE graph has two components, while WHITHER has only one.
Links with n-sway at least 100 are highlighted. 
Except for \texttt{CA} $\leftrightarrow$ \texttt{GB} on the right, all of these have 
\texttt{US} as an endpoint. }}

\label{f:migration}
\begin{center}
\scalebox{0.34}{\includegraphics{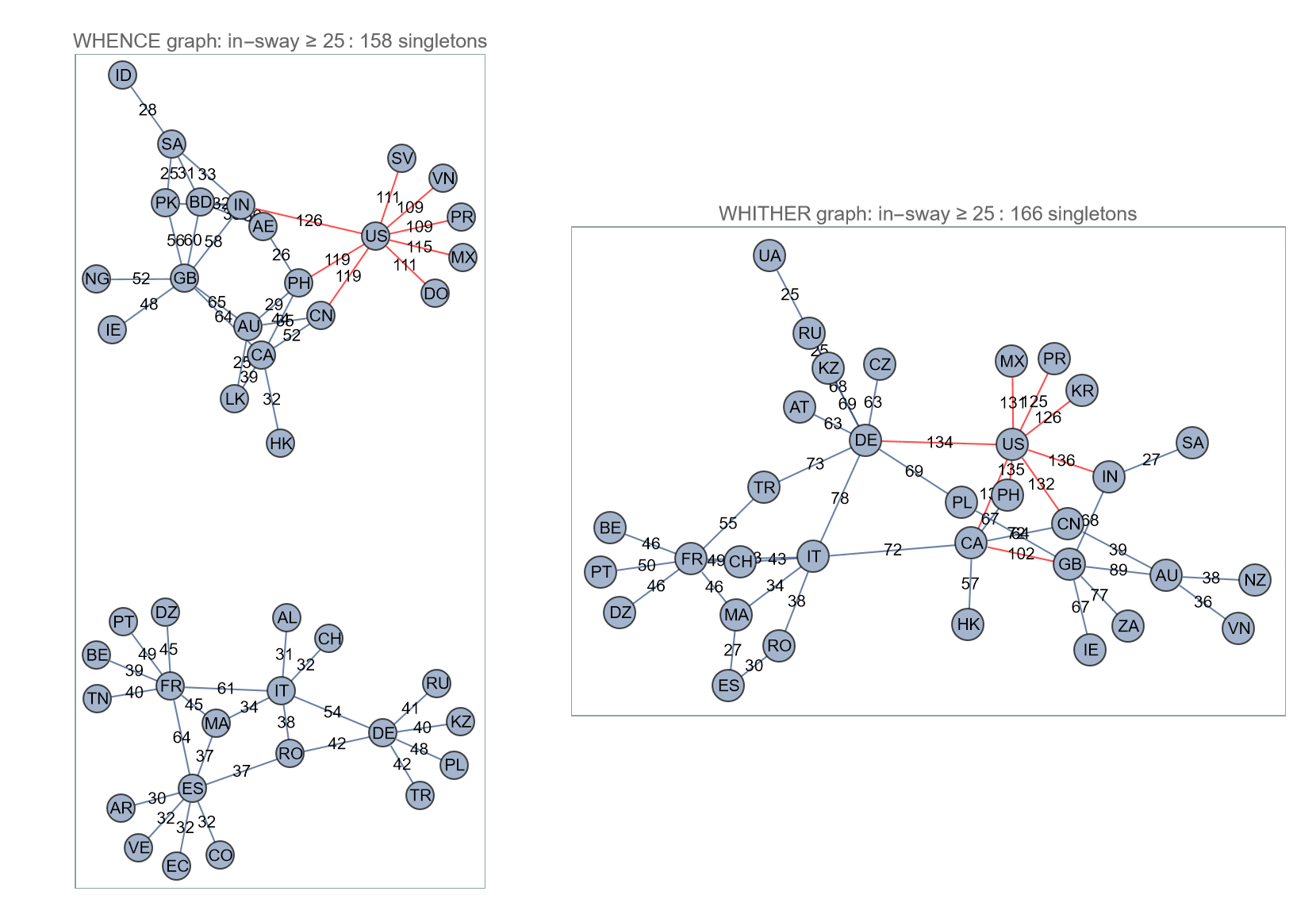} }
\end{center}
\end{figure}

\subsubsection{Digraph example: migration flows among 200 countries} \label{s:migration}
Azose \& Raftery \cite{azo} build a statistical model, based on OECD data, to estimate how many persons
who lived in country $x$ in 2010 were living in country $y$ in 2015. They take account of both migration and
reverse migration, as distinct from net migration often seen in official figures. Starting from their $200 \times 200$
table of flows in  \cite{azo}, we selected flows of size at least 1000 persons, and built a sparse directed weighted graph
with 3866 edges. Our Java code \cite{rbl} found the 2-core (199 countries), then 8 nearest neighbors of each country under 
each of two comparators:
\begin{enumerate}
  \item[(a)]
  WHENCE comparator: $y \prec_x z$ if flows from country $y$ to country $x$ exceed flows from country $z$ to $x$.
  \textit{Whence are immigrants coming from?}
    \item[(b)]
    WHITHER comparator: $y \prec_x z$ if flows from country $x$ to country $y$ exceed flows from $x$ to 
country $z$.\textit{Whither are our citizens emigrating?}
\end{enumerate}
Using each comparator, a portion of the linkage graph is shown in Figure \ref{f:migration}.
Country digraphs are listed at \url{immigration-usa.com/country_digraphs.html}.

Switching between a source-base comparator and a target-based comparator is an option
for rank-based linkage on any weighted digraph.

\subsubsection{Run times on larger graphs}
Table \ref{tab:runtimes} shows some run times for our initial un-optimized Java implementation \cite{rbl}, with
JDK 19 on dual Xeon E3-2690 (48 logical cores).
\begin{table}[]
    \centering
    \begin{tabular}{c|c|c|c|c|c}
       $S$  &  $K$ & Digraph prep. & Rank-based linkage & Components & RAM \\ \hline
       1M.  &  8   & 16 secs & 5 secs & 5 secs & 8 GB\\ \hline
    10M.    &  8   & 228 secs & 44 secs & 49 secs & 78 GB\\ \hline
    \end{tabular}
    \caption{\textit{
Albert-Barabasi random graphs with parameters $(|S|, 2)$ were generated, and uniform random weights assigned to
the 2M and 20M edges, respectively, for the cases shown. The fourth column is consistent with rank-based linkage run time scaling linear in $|S|$, for fixed $K$. The final column confirms RAM usage scales in the same way.
    }  
    }
    \label{tab:runtimes}
\end{table}

%%%%%%%%%%%%%%%%%%%%%%%%%%%%%%%%%%%%%%%%%%%
\section{Sheaves on a simplicial complex} \label{s:asc}

\subsection{Geometric intuition}

Sheaves \cite{cur} on a simplicial complex supply a toolbox for precise description of rank-based linkage. 
Figures \ref{f:k4} and \ref{f:linek4} supply geometric intuition. The six points, twelve edges,
and eight triangles of the octahedron shown in Figure \ref{f:linek4}
are the  0-cells, 1-cells, and 2-cells, respectively of an
abstract simplicial complex. The 3-concordant ranking system shown in Figure \ref{f:k4} leads to an orientation of the 1-cells such that none of the
2-cells is a directed 3-cycle. This orientation of the 1-cells provides
a toy example of an \textit{oriented simplicial complex}.

Including all the points, edges, and triangles of the line graph
of the complete graph of $S$ leads to an excessive computational
load for large $n$. The next section outlines a parsimonious approach.

\subsection{A 2-dimensional simplicial complex based on a line graph}\label{s:complex}

Take a graph $G$ on the set $S$, 
and then take a subgraph $H$ of the line graph $L(G)$,
whose vertex set $V(H)$ is all the edges of $G$.
There is a natural way to build a two-dimensional abstract simplicial complex
whose 0-cells, 1-cells, and 2-cells are the points,
edges, and triangles, respectively, of $H$.
To be precise, $\Delta(H) = \Delta_0 \cup \Delta_1 \cup \Delta_2$
is defined as follows:

\begin{enumerate}
    \item[0:] The 0-cells or 0-simplices $\Delta_0$ of $\Delta(H)$ are the edges of $G$. A typical element of $\{x, y\} \in \Delta_0$ will be abbreviated to $xy$, for distinct adjacent $x, y \in S$.

    \item[1:] The 1-cells or 1-simplices $\Delta_1$ are
     the edges of $H$,
i.e. pairs $\{xy, xz\} \in H$ for distinct 0-cells $xy$ and $xz$
    \footnote{
    Note that $x_1 x_2$ and $x_3 x_4$ are non-adjacent in the line graph,
    for distinct $\{x_i\}$.
    }. 
    
    \item[2:] The 2-cells or 2-simplices $\Delta_2$ of $\Delta$ are
    all triples of pairwise intersecting $1$-cells. Thus $2$-cells
    are precisely the triangles of $H$.
   There are two types of $2$-cells:
    \begin{itemize}
     \item $\{xy, yz, xz\}$ for distinct $x, y, z$ in $S$ (``voter triangles'' as in (\ref{e:voter3}).
        \item $\{vx, vy, vz\}$ for distinct $v, x, y, z$ in $S$ (``enforcer triangles''
        as in (\ref{e:enforcer3}).
    \end{itemize}

\end{enumerate}

    When $G$ is the complete graph $\mathcal{K}_S$ on $S$, 
    and $H$ is the line graph $L(\mathcal{K}_S)$,
    we obtain the points,
edges, and triangles, respectively, of this line graph.
This complex imposes at least a cubic computational burden, 
because $\Delta_1$ has $n(n-1)(n-2)/2$ elements.

Recall that an orientation of a graph means an assignment of a direction to each edge.
An orientation of the simplicial complex $\Delta(H)$ means
an orientation of $H$.

We repeat, in simplicial complex notation, the definitions of Section \ref{s:crs}.

\begin{defn}\label{d:3-conc}
A ranking system on $S$ is an orientation of the $1$-cells of $\Delta(L(\mathcal{K}_S))$ (i.e. edges of the line graph)
such that 2-cells of the form $\{vx, vy, vz\}$ are acyclic, for distinct $v, x, y, z$ in $S$. If 2-cells of the form $\{vx, vy, xy\}$ are acyclic also, call the ranking system 3-concordant. 
Call it concordant if the orientation of the entire line graph
is acyclic.
\end{defn}

Figures \ref{f:linek4} and \ref{f:counterex} illustrate 
concordant and 3-concordant ranking systems, respectively.

\subsection{Oriented simplicial complex associated with out-ordered digraph} \label{s:scprt}
Some care is needed when adapting the pattern of Section \ref{s:complex} to the case of
an out-ordered digraph $(S,\Gamma,\preceq)$  (Definition \ref{d:outorddigr}).
Take the graph $G$ of Section \ref{s:complex} to be
the undirected $K$-NN graph (Definition \ref{d:unng}),
denoted $(S, \mathcal{U}_{\Gamma})$.
Recall that the line graph $L(G_{\Gamma})$ of Definition \ref{d:lingraph} is deprecated because some of its 1-simplices may be unorientable.
Focus instead on the \textit{based line graph} $H_{\Gamma}$ (Definition \ref{d:basedlg}) designed to avoid this problem.
 
Define an oriented simplicial complex 
\begin{equation}\label{e:ascknng}
  \Delta (H_{\Gamma}):=\Delta_0  \cup \Delta_1  \cup \Delta_2  
\end{equation}
whose simplices of dimension 0, 1, and 2 are as follows:
\begin{enumerate}

    \item[0:] The 0-cells $\Delta_0 = \mathcal{U}_{\Gamma}$,
   i.e. edges of the undirected $K$-NN graph.
    
    \item[1:] The 1-cells $\Delta_1$ are
     pairs $\{xy, xz\} \in \binom{\mathcal{U}_{\Gamma}}{2}$ 
such that $\{y, z\} \cap \Gamma(x) \neq \emptyset$,
with an orientation acquired from the 
out-ordered digraph $(S,\Gamma,\preceq)$ using the formula (\ref{e:1-simplices}) when $|\{y, z\} \cap \Gamma(x)| = 1$.

    \item[2:] The 2-cells $\Delta_2$ are triples of 0-cells,   
    all of whose 1-faces are in $\Delta_1$ (see Definition \ref{d:strictkpert}
    for consequences). If the ranking system
    is 3-concordant, none of these 2-cells is cyclic according
    to the orientation of its 1-faces.
    
\end{enumerate}
Rank-based linkage computations (Algorithm \ref{a:rbl}) make reference to this oriented simplicial complex,
without explicitly constructing it.

\begin{defn}\label{d:ood3concord}
 An out-ordered digraph $(S,\Gamma,\preceq)$ is called 3-concordant (resp. concordant) if the induced
 orientation of $\Delta(H_{\Gamma})$ is 3-concordant (resp. concordant).
\end{defn}

\subsection{Sheaf perspective}\label{s:sheafview}
Let $\Lambda$ be any abstract simplicial complex
(for example $\Delta(H_{\Gamma})$ above,
but the restriction to 0-cells, 1-cells, and 2-cells seems unnecessary).
We recall the standard definition of a topology on $\Lambda$, 
as presented by Steiner \cite{ste} and Stong \cite{sto}; we employ the version in which
the lower sets are open, rather than the upper sets.

\begin{comment}
\adam{There must be a convenient reference.  All I have is Lee's book on
topological manifolds and Hatcher's book on algebraic topology, neither of
which gets into this.  I found an old paper of McCord online.
Also, we seem to be using the version where the lower sets are open, rather
than the upper sets, which may be more popular nowadays.}
\end{comment}
\begin{defn}\label{d:alexandrov-topology}\cite{mccord}
The {\em Alexandrov topology} on $\Lambda$ is the topology whose open sets
are the sets of simplices contained in a fixed $x \in \Lambda$.
\end{defn}
(More generally the Alexandrov topology is defined on an arbitrary partially
ordered set $S$ by taking the $\{x \in S: x \le x_0\}$ to be the basic
open sets.)

\begin{ex}\label{ex:2-cell}
Let us consider the $2$-cell $\{a, b, c\}$, a 
simplicial complex with $7$ simplices.  Rather than write these simplices
as $\{a,b\}$, etc., we abbreviate this to $ab$.  The basic open set
associated to the $2$-cell is the entire space
\[
U:=\{a, b, c, ab, bc, ac, abc\}
\]
(we have not introduced any orientation so there is no distinction between
$ac$ and $ca$).
Contained within $U$ are the basic open sets $\{a, b, ab\}$,
$\{a, c, ac\}$, and $\{b, c, bc\}$, as well as the singletons
$\{a\},\{b\},\{c\}$, but these do not cover $U$.  Note that the point
$abc$ is dense in the space; in general there is a unique minimal open cover which
consists of the closures of the maximal simplices.
\end{ex}

\begin{defn}\label{d:loop}
Let $\Lambda$ be an abstract simplicial complex.
A {\em $k$-loop} is a set of simplices
$\{a_0 a_1, a_1 a_2, \ldots, a_{k-1} a_0\}$ of $\Lambda$,
where the $a_i$ are distinct $0$-simplices.
Equivalently, it is the image of an injective map from the standard
$k$-loop, which is the simplicial complex whose $0$-cells are 
$\{0,1,\dots,k-1\}$ and whose $1$-cells are $\{01,12,\dots,(k-1)0\}$.
\end{defn}
For example, given any face $abc$, the set $\{ab, bc, ca\}$ is a $3$-loop.

\begin{defn}\label{d:orientation-sheaf}
Let $\Lambda$ be a simplicial complex, with $\Lambda_1$ as its 1-simplices.
An {\em orientation} on a $1$-simplex $\lambda$ is one of the $2$ orderings
of its set of vertices, and the set of orientations on $\lambda$ will be
denoted $\pi(\lambda)$.
The {\em orientation sheaf} $\OO$ on $\Lambda$ is the sheaf of sets 
defined by
\[
{\OO}(U) = \prod_{\lambda \in U \cap \Lambda_1} \pi(\lambda),
\]
with restriction maps given by the natural maps from larger to smaller
cartesian products.
\end{defn}

\begin{prop}\label{prop:orientations-are-sheaves}
The orientation sheaf just defined is a flasque sheaf of sets.
\end{prop}
% Richard - Ops - you are right, of course! \adam{I thought that separation was part of the definition of a sheaf?}

\begin{proof} Let $U \subseteq \Lambda$ be an open subset and $\{V_i\}$
an open cover; let $U_1$ be the set of $1$-simplices of $U$.  Fix 
$u \in U_1$.  Then $u$ must belong to at least one of the $V_i$.
Let $s_i \in {\OO}(V_i)$ be a collection of sections
that agree on the $V_i \cap V_j$.  Then all $s_i$ with $u \in V_i$
must take the same value on $u$, say $\rho(u)$.  Define $s \in {\OO}
(U)$ to be the section that takes $u$ to $\rho(u)$ whenever $u \in V_i$.  It is clear that
$s|_{V_i} = s_i$ for all $i$ and that $s$ is the unique section with this
property, so $\OO$ is a sheaf, and it is flasque from the definition.
\end{proof}

\begin{comment}
\begin{defn}\label{d:careful-sheaf}
Let $\mcf'_3$ be the presheaf such that $\mcf'_3(U)$ is the set of orientations of the 
$1$-simplices of $U$ that are $3$-acyclic in the strong sense that $U$ does not contain vertices $a, b, c$ and
edges $a \rightarrow b$, $b \rightarrow c$, $c \rightarrow a$ oriented as indicated, even if the $2$-simplex
$abc$ does not belong to $U$.  
More generally, let $\mcf'_m(U)$ be the set of orientations
of the 1-simplices of $U$ such that all $k$-loops of $1$-simplices of
$U$ are acyclic, for all $k = 3,4,  \ldots, m$.  For $m = \infty$ we interpret this notation as the set of concordant orientations; for $m = 2$, as the set of all orientations.
\end{defn}
\end{comment}

\begin{defn}\label{d:k-acyclic}
Let $s \in \OO(U)$ be a section of the orientation sheaf over an
open set, and let $k \in \N \cup \{\infty\}$.  
If there is no $m$-loop $a_0a_1, \dots, a_{m-1}a_0$ in $U$ for any 
$3 \le m \le k$ such that $s$ assigns to each $a_ia_{i+1}$
(including $a_{m-1}a_0$)
the order $a_i < a_{i+1}$, then $s$ is {\em $k$-acyclic}.
\end{defn}

Note that the $m$-loops are not required to bound a $2$-simplex or a 
$2$-dimensional simplicial complex.  Also, the condition is vacuous if
$k = 2$.

\begin{defn}[Strong presheaf]\label{d:careful-sheaf}
For $k \ge 2$ as above and any simplicial complex,
define $\mcf'_k$ to be the subpresheaf of the orientation sheaf such that
$\mcf'_k(U)$ is the set of $k$-acyclic orientations of $U$.
\end{defn}

In other words $\mcf'_k(U)$ means the set of orientations so that no directed cycle of length $\leq k$ occurs anywhere in 
the 1-simplices of $U$.
A subpresheaf of a separated presheaf is separated; by definition
$\mcf'_k \subseteq \OO(\Lambda)$ which is separated, so this applies to $\mcf'_k$.
We give some examples of this construction.

\begin{ex}
    \begin{enumerate}
        \item[(a)] Let $\Lambda$ be the $2$-cell $abc$.  Then
        $\#(\mcf'_3(\Lambda)) = 6$.  Indeed, of the $8$ elements of
        $\OO(\Lambda)$, all are sections of $\mcf'_3(\Lambda)$ except for
        $(ab,ca,bc)$ and $(ba,ac,cb)$.
        \item[(b)] If $\Delta = \Delta(L(\mathcal{K}_S))$ as in Definition \ref{d:3-conc}, then $\mcf'_3(\Delta)$ (resp. $\mcf'_n(\Delta)$) 
        coincides with the 3-concordant  (resp. concordant) ranking systems on $S$.
        \item[(c)] For all $k$ and all $\Lambda$, the set
        $\mcf'_k(\emptyset)$ has $1$ element.
    \end{enumerate}
\end{ex}

\begin{comment}
\begin{defn}\label{d:careful-sheaf}
Let $k \in \N \cup \{\infty\}$.  We
define $\mcf'_k$ to be the 
$U \mapsto \mcf'_m(U)$ is a separated presheaf with respect to the
topology on $\Lambda$, because the restriction map
$ \mcf'_m(W) \to  \mcf'_m(U)$ is the composite
$\mcf'_m(W) \to \mcf'_m(V) \to \mcf'_m(U)$
when $U \subseteq V \subseteq W$ and two sections that agree on an
open cover are equal.
Examples of this construction include:

\begin{enumerate}
    \item[(a)] If $U$ is the basic open set associated with a 2-cell, there are six elements in
    $\mcf'_3(U)$; out of $2^3 = 8$ orientations of the
    three 1-simplices, all but two are acyclic.
    
    \item[(b)] 
    If $\Delta = \Delta(L(\mathcal{K}_S))$ as in Definition \ref{d:3-conc}, then
    $\mcf'_3(\Delta)$ (resp. $\mcf'_n(\Delta)$) coincides with the 3-concordant 
    (resp. concordant) ranking systems on $S$.

 \item[(c)] $\mcf'_n(\emptyset)$ is formally a 1-element set for all $n$.
\end{enumerate}
\end{defn}
\end{comment}

\begin{rem}\label{r:not-sheaf}
It is easy to see that $\mcf'_k$ is not a sheaf for $k \ge 3$, by considering a simplicial complex with
vertices $x_0, \dots, x_{k-1}$ and edges $E_i: x_i x_{i+1}$ for $0 \le i < k$ with indices read mod $k$.  If the edges are oriented by $x_i \to x_{i+1}$ to form a $k$-loop, then the obvious sections on the open sets containing $E_i$ and its endpoints do not glue to any section of $\mcf'_k$.  Other 
natural examples appear in Figure~\ref{f:noglue3} for $k = 3$ and in Figure~\ref{f:sheafcounterex} for $k = 4$.  
\end{rem}
%\adam{This figure does correctly illustrate that $\mcf'_4$ is not a sheaf, but in a more complicated way than described in the previous sentence.}
%% k=3 DIAGRAM %%%%%%%%%%%%%%%%%%%%%%%%%%%%%
\begin{figure}[t]
\caption{\textbf{$\mcf_3'$ is not a sheaf: }
\textit{The seven (locally triangle-acyclic) sections agree when
restricted, and glue to the whole, but the glued section has an outer directed
3-cycle (in red), which breaks the sheaf condition for $\mcf_3'$. Four of the 2-cells are enforcers, while four (including red) are voters.
(Courtesy of the referee)}
}
\label{f:noglue3}
\begin{center}
\begin{tikzpicture}[
    node distance=3cm and 4cm,
    every node/.style={font=\large\itshape, color=gray},
    every path/.style={->, >=stealth, thick, draw}
]
% Define nodes
\node (cd) at (0,0) {cd};
\node (bc) at (8,0) {bc};
\node (ac) at (4,-2.5) {ac};
\node (ad) at (2,-5) {ad};
\node (ab) at (6,-5) {ab};
\node (db) at (4,-7.5) {db};

% Draw straight arrows (gray)
\draw[gray] (cd) -- (ac);
\draw[gray] (cd) -- (ad);
\draw[gray] (cd) -- (ad);
\draw[gray] (bc) -- (ac);
\draw[gray] (bc) -- (ab);
\draw[gray] (ad) -- (ac);
\draw[gray] (ab) -- (ac);
\draw[gray] (ab) -- (ad);
\draw[gray] (db) -- (ad);
\draw[gray] (db) -- (ab);

% Draw curved arrows (red)
\draw[red, bend right=30] (bc) to (cd);
\draw[red, bend right=30] (cd) to (db);
\draw[red, bend right=30] (db) to (bc);

\end{tikzpicture}
\end{center}
\end{figure}

%\begin{figure}
%\caption{\textbf{$\mcf'_3$ is not a sheaf: }
%\textit{An open set U consists of 2-simplices $\{axy, byz, cxz\}$, and all included 0- and 1-faces.
%This is shown as three triangles, each pair having a 0-simplex in common.
%All three 2-simplices are acyclic, but
%$x\rightarrow y\rightarrow  z\rightarrow  x$ is a 3-cycle,
%which is allowed because $xyz$ is NOT one of the 2-simplices.
%In that case U contains a 3-cycle, but none of its open subsets do; so gluing fails.
%The failure of the gluing axiom for
%the pre-sheaf $\mcf'_3$ shows it cannot be a sheaf.
%}
%}
%\label{f:sheaf3counterex}
%\begin{center}
%\scalebox{0.35}{\includegraphics{IMAGES/3-concordant-sheaf-counterexample.png} }
%\end{center}
%\end{figure}

\subsection{3-concordant ranking systems form a sheaf}
We now show how to modify the definition of $\mcf'_3$ so as to define a sheaf $\mcf_3$.
The reader who has grasped Algorithm \ref{a:rbl} but wonders whether the sheaf property is
an unnecessary luxury may look ahead to Section \ref{s:sheafvalue}.
Unfortunately it does not appear that there is any analogous modification of $\mcf'_k$ that would allow
us to define a sheaf of $k$-concordant orientations on a simplicial complex of dimension $2$ when $k \geq 4$.

A 2-cell $\tau \in U$ has a boundary $\partial \tau$ consisting of three 1-cells. Denote by $\text{Acyc}(\partial \tau)$ the
six acyclic orientations of the boundary $\partial \tau$ (a subset of a subproduct of $\OO(U)$).

\begin{defn}[Weak subsheaf]\label{d:presheafcyclefree}
Define the presheaf $\mcf_3 \subseteq \OO$ by
\[
\mcf_3(U) = \left\{
\alpha \in \OO(U)\mid \forall \tau \in \text{2-simplices of } \, U, \quad
\alpha|_{\partial \tau} \in \text{Acyc}(\partial \tau)
\right\}
\]
\end{defn}

\begin{ex}\label{ex:clarify-sheaf} 
Definition \ref{d:presheafcyclefree} does
not impose any condition on edges $ab, bc, ac$ if $abc$ is not a $2$-simplex of $U$.
For example, suppose that $U_1$ is the
$1$-skeleton of a $2$-simplex $U$.  Then $\mcf_3(U)$ and $\mcf'_3(U)$ agree, since they
equal the $6$-element
set of acyclic orientations of the $1$-faces of $U$.
However, $\mcf_3(U_1)$ and $\mcf'_3(U_1)$ differ: $\mcf'_3(U_1) = \mcf'_3(U)$, while
$\mcf_3(U_1)$ is the $8$-element set of all orientations of the $1$-faces of $U$.
\end{ex}

\begin{ex}\label{ex:8-triangles} 
Figure \ref{f:noglue3} does not exhibit a gluing failure for $\mcf_3(U)$ if the 2-simplex $\{bc, cd, db\}$ (shown in red) is not included in $U$.
\end{ex}

\begin{comment}
 Thanks to the referee for the observation that $\mcf_3$ is the equalizer subsheaf 
 \[
\OO \longrightarrow \prod_{\tau \in (-)_2} \OO (\partial \tau) \rightrightarrows  \prod_{\tau \in (-)_2} \text{Acyc}(\partial \tau)
 \]
 where $(-)_2$ takes an open set $U$ to its set of 2-simplices $U_2$.
\end{comment}

\begin{rem}
The definition of $\mcf_3(U)$ prepares for the selection in Definition \ref{d:strictkpert} of the class of voter triangles
which are called `` $\Gamma$-pertinent''. The subsequent Lemma \ref{l:kpert} shows that for orientations in $\mcf_3(U)$,
there cannot exist triples $\{xy, yz, zx\}$ where $y$ (but not $z$) is a friend of $x$,
$z$ (but not $x$) is a friend of $y$, and $x$ (but not $y$) is a friend of $z$ -- a situation which offers no insight into which pair out 
of objects $\{x, y, z\}$ is the most closely related.
\end{rem}

Definition \ref{d:presheafcyclefree}  generalizes to any $k \geq 3$. Whereas the ``strong'' presheaf $\mcf'_k(U)$ of
Definition \ref{d:careful-sheaf} demands no cycles of length up to $k$,   
$\mcf_k(U)$ requires only that all simplices of $U$ of dimension less than $k$ are acyclic.

\begin{defn}\label{d:presheafcyclefree-gen}
Let $k \in \N \cup \infty$.
Define the presheaf $\mcf_k \subseteq \OO$ by
\[\mcf_k(U) = \left\{\alpha \in \OO(U)\mid 
\alpha|_{\partial \tau} \in \mcf'_\infty(\partial \tau)\,\text{for every $\tau$ which is a $k-1$-simplex of $U$}\right\}.
\]
\end{defn}

%%%%%%%%%%%%%%%%%%%%%%%%%%%%%%%%%%%%%%%%%%%%%%%%%%%%%%%%%%%%%%%%%%
\begin{figure}[b]
\caption{\textbf{$\mcf_4'$ is not a sheaf: }
\textit{The left and center digraphs show 4-concordant oriented simplicial complexes whose orientations coincide
on their common 2-simplices. Their union, which includes 1-simplices $b \rightarrow c$
and $c \rightarrow a$, contains both a 4-cycle, shown in red, and a 5-cycle. The failure of the gluing axiom for
the pre-sheaf $\mcf_4'$ shows that it cannot be a sheaf.
}
}
\label{f:sheafcounterex}
\begin{center}
\scalebox{0.45}{\includegraphics{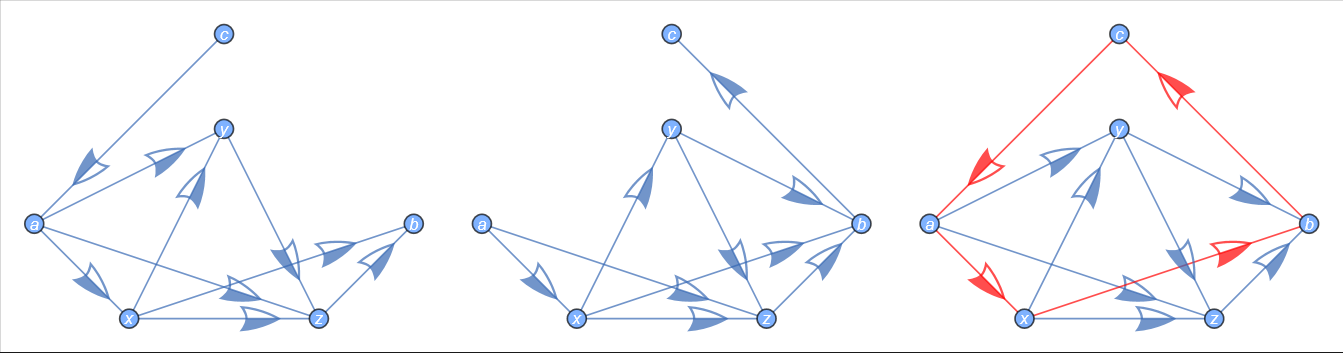} }
\end{center}
\end{figure}

\begin{comment}
\begin{defn}
For any open set $U$ let $\mcf_3(U)$ 
be the set of orientations
of the 1-simplices of $U$ that are acyclic
%\footnote{Prasad Tetali has coined the term ``tricycle-free''.} 
on all 2-simplices of $U$,
i.e., with no 2-simplex $abc$ whose 1-faces have
orientations $a \rightarrow b$, $b \rightarrow c$, $c \rightarrow a$.
\end{defn}

\begin{defn}\label{d:presheafcyclefree-gen}
More generally, for $k$ a positive integer let $\mcf_k(U)$ be the set of orientations of the $1$-simplices of $U$ that are acyclic on all simplices of $U$ of dimension less than $k$.
\end{defn}
\end{comment}

In Proposition~\ref{p:sheafify} we will show that $\mcf_k(U)$ is a sheaf.

\begin{rem}\label{rem:stop-at-dim}
If $k > \dim U$ then $\mcf_k(U) = \mcf_{\dim U + 1}(U)$.  Also, for $k \le 2$ every orientation of the $1$-simplices of $U$ belongs to $\mcf_k(U)$.
\end{rem}

%Other results such as Lemma \ref{l:pullback} also apply under these modified definitions---we still get rank-labelled nearest-neighbour digraphs and so on.

\begin{prop}\label{p:sheafify}
$\mcf_k$ (Definition \ref{d:presheafcyclefree-gen}) is the sheafification of $\mcf'_k$ and is therefore a sheaf.
\end{prop}

\begin{proof}
Let $U$ be an open set.  The sections of the sheafification of $\mcf'_k$ over $U$
are compatible collections of sections on open covers of $U$ mod the relation that 
identifies pairs of such collections that agree after further refinement.  In particular,
let $S_1, \dots, S_n$ be the maximal simplices contained in $U$.  Every open cover
of $U$ must contain at least one open set that includes $S_i$ and therefore its edges;
thus the orientation must have no cycles of length $k$ or less on the edges of $S_i$.  On the other hand, by
considering an open cover with one set for each simplex in $U$, we see that there are
no further conditions.

Clearly a section of the sheafification over $U$ determines an orientation of the
$1$-simplices in $U$ uniquely (that is, equivalent collections of sections on open
covers determine the same orientation), and conversely an orientation with no cycles of length $k$ or less on the
$S_i$ determines a section over $U$.  We thus see that, for every $U$, the sections
of $\mcf_k$ and of the sheafification of $\mcf'_k$ are in canonical bijection.  It is
obvious that the restriction maps match as well.
\end{proof}

\begin{prop}\label{prop:same-if-too-big}
Let $\Lambda$ be a simplicial complex and $k \in \N \cup \{\infty\}$.
For all open sets $U \subseteq \Delta$ we have
$\mcf_k(U) = \mcf_{\min(k,\dim U + 1)}(U)$.
\end{prop}

\begin{proof} Indeed, there are no simplices in $U$ whose dimension is
greater than $\dim U$, so $\mcf'_k(U) = \mcf'_{\min(k,\dim U+1)}(U)$.
The result then follows for the sheafifications as well.
\end{proof}

\begin{cor}\label{cor:not-too-many}
There are at most $\dim \Lambda$ different sheaves among the $\mcf_k$.
\end{cor}

\begin{proof} By the proposition, we have $\mcf_k = \mcf_{\dim \Lambda+1}$
if $k \ge \dim \Lambda+1$.  It is clear that $\mcf_k = \mcf_2$ if $k \le 2$.
\end{proof}

Again we recall that our primary interest is in simplicial complexes of dimension $2$.  We state a corollary applicable to this setting.

\begin{cor}\label{dim-2-always-f3}
Let $\Lambda$ be a simplicial complex of dimension at most $2$.
Then $\mcf_k(U) = \mcf_3(U)$ for all $k \ge 3$.  
\end{cor}

%%% This is now unnecessary, because the sheafification is always a sheaf.
%\begin{prop}[Gluing property]\label{p:sheaf}
%$U \mapsto \mcf_3(U)$ (Definition \ref{d:presheafcyclefree})
%is a sheaf on the open sets of any 2-dimensional abstract simplicial complex.
%\end{prop}

%\textbf{Remark: } The assertion is generally false for $\mcf_m$ when $m \geq 4$,
%as shown in Figure \ref{f:sheafcounterex}.
%The proof of Proposition \ref{p:sheaf} works because every 3-loop may be cast as a 2-simplex,
%whereas a 4-loop (with four 1-simplices) cannot be cast as a 3-simplex (with six 1-simplices).

%\begin{proof}
%Let $\{U_i\}$ be an open cover of an open set $U$ in a 2-dimensional abstract simplicial complex.
%We show that $\mcf(U)$
%is the subset of $\prod_i \mcf(U_i)$ consisting of sections that
%restrict consistently to $\prod_{i,j} \mcf(U_i \cap U_j)$.
%Indeed, given a 1-simplex in $U$, all $U_i$ containing it assign it the same
%orientation, so we can define a unique orientation on the set of
%1-simplices of $U$ compatible with those on the $\{U_i\}$. Given a $2$-simplex
%in $U$, it must be contained in some $U_i$, and the orientations of its
%1-simplices in such a $U_i$ are acylic, and so they are acyclic on $U$.
%Every 3-loop inside $U$ may also be cast as a 2-simplex, and as such it must be contained
%in at least one of the $U_i$, where it is acyclic by assumption.
%\end{proof}

\begin{cor}\label{c:3sheaf}
The 3-concordant orientations of the 
2-dimensional abstract simplicial complex $\Delta (H_{\Gamma})$ (see (\ref{e:ascknng})
form a sheaf with respect to the topology on $\Delta (H_{\Gamma})$.
\end{cor}

\textbf{Interpretation: }The 3-concordant out-ordered diagraphs $(S, \Gamma, \preceq)$
(Definition \ref{d:ood3concord}), for fixed $S$ and
friend sets $\{\Gamma(x), x \in S\}$, form a sheaf in the sense of Corollary \ref{c:3sheaf}.

\subsection{How the sheaf property affects the rank-based linkage algorithm}\label{s:sheafvalue}
The primary effect of Definition \ref{d:presheafcyclefree} of the sheaf $\mcf_3$ (proved as such in Corollary \ref{c:3sheaf}),
appears in the focus in the next section (Definition \ref{d:strictkpert} and Lemma \ref{l:kpert}) on fixing in advance the collection of 
2-simplices on which acyclicity will be enforced. If one were to choose an approach based on  $\mcf'_3$ (Definition \ref{d:careful-sheaf})
then directed 3-cycles could crop up ``out of the blue'' when forming an open cover, as Figure \ref{f:noglue3} demonstrates;
unanticipated 3-cycles would interfere with the counting process in Algorithm \ref{a:rbl}.

We prohibit directed 3-cycles in order to compute unambiguously the in-sway between a pair of objects, by counting 2-simplices in which 
this pair is the source in a directed acyclic triple. We are at liberty to permit cycles of length four and above, since they have 
no direct algorithmic impact. Working with 3-concordant ranking systems, rather than fully concordant ones has two great benefits: (1) 
the computational effort to verify 3-concordance is much less than that for full concordance; and (2) 
concordant ranking systems arise when objects have a metric embedding (well-studied in data science), whereas 
3-concordant ranking systems can occur when no such metric embedding is possible -- see Baron \&Darling \cite{bar}.

Sheaf morphisms make a formal appearance in Section \ref{s:gluecompare};
we conclude in Section \ref{s:sheafmerge} that sheaves offer greater generality than ranking systems in pooling data.

\section{Rank-based linkage based on the K-nearest neighbor digraph}\label{s:rblformal}

\subsection{Rank-based linkage calculations}
 Algorithm \ref{a:rbl} performs a summation over $y \in S$ of voter triangles 
 $\{xy, yz, xz\} \in \Delta_2$ as defined in (\ref{e:ascknng}). 
Definition \ref{d:strictkpert} makes membership of $\Delta_2$ precise.

\begin{defn}\label{d:strictkpert}
Call a 2-simplex $\{xy, yz, xz\}$ (i.e. a voter triangle) $\Gamma$-pertinent if
 all three pairs $\{x, y\}$, $\{y, z\}$,  and $\{x, z\}$ are in $\mathcal{U}_{\Gamma}$, and moreover all three sets
 \[
 \{y, z\} \cap \Gamma(x),  \quad \{x, y\} \cap \Gamma(z),
 \quad \{x, z\} \cap \Gamma(y)
 \]
 are non-empty.
\end{defn}

A 2-simplex which is not $\Gamma$-pertinent plays no role in
Algorithm \ref{a:rbl}. For example 2-simplices $\{xy, yz, xz\}$
in which neither $x$ nor $z$ is a friend of $y$ are ignored, even when
$z \prec_x y$ and $x \prec_z  y$. The ``third wheel'' social metaphor
for this choice is discussed in Section \ref{s:linkagefunctor}. Here is a geometric explanation.

\textbf{Point cloud example: } Consider objects as points in a 
Euclidean space with a Euclidean distance comparator. 
Suppose $x$ and $z$ are outliers which are closer to each other than to the origin,
and the other $n-2$ objects are clustered around the origin.
In that case $\{x, z\} \cap \Gamma(y)$ will be empty
for all $y$ near the origin, and the 0-simplex $xz$ will not be
a 0-face of any $\Gamma$-pertinent 2-simplices.
Rank-based linkage will give the pair $\{x, z\}$ zero in-sway, making them isolated points,
not a cluster of two points.

Lemma \ref{l:kpert} provides the rationale for the organization of Algorithm \ref{a:rbl}
in the 3-concordant case, and leads to Proposition \ref{p:kpertbound},
which bounds the number of $\Gamma$-pertinent 2-simplices.

\begin{lem}\label{l:kpert}
Assume the ranking system is 3-concordant.
In every $\Gamma$-pertinent 2-simplex $\{xy, yz, xz\}$, there is at least
one pair of mutual friends among the objects $\{x, y, z\}$.
\end{lem}

\begin{proof}
Consider the subgraph of the directed $K$-NN graph (Definition \ref{d:riwdg})
induced on the vertex set $\{x, y, z\}$.
Definition \ref{d:strictkpert} shows that each of the objects $\{x, y, z\}$
has out-degree at least one in this subgraph. 

If at least one object has out-degree two, then the 
total in-degree of the subgraph is at least four.
Since there are three vertices, the pigeonhole principle shows
that some vertex has in-degree two, say vertex $x$. Then $x$ is the target
of arcs from both $y$ and $z$. However $x$
is also the source of at least one arc, say to $z$. Hence
$x$ and $z$ are mutual friends.

On the other hand if each of $\{x, y, z\}$ has in-degree one and out-degree one, as in
 \[
      \begin{tikzcd}{} & x \arrow{dr}{} \\y \arrow{ur}{}  && 
      \arrow{ll}{}z \end{tikzcd} 
 \]
then (\ref{e:1-simplices}) implies that 
the 2-simplex $\{xy, yz, xz\}$ is a directed 3-cycle, in this case:
  \[
      \begin{tikzcd}{} & yz \arrow{dr}{} \\xy \arrow{ur}{}  && 
      \arrow{ll}{}xz \end{tikzcd} 
 \]
which violates the assumption of 3-concordance.
Hence the number of arcs in the subgraph cannot be three,
which completes the proof.
\end{proof}

\begin{prop}\label{p:kpertbound}
    In a 3-concordant ranking system on $n$ objects, the
     number of $\Gamma$-pertinent 2-simplices does not exceed
    $n K^2$.
\end{prop}
\begin{proof}
Given a $\Gamma$-pertinent 2-simplex $\{xy, yz, xz\}$,
we may assume by Lemma \ref{l:kpert} that $x$ and $z$ are mutual friends,
and $x \in \Gamma(y)$, after relabelling if necessary.
There are $n$ choices of $y$, $K$ choices of $x$, and at most $K$
choices of $z$ in $\Gamma(x) \setminus \{y\}$, giving
$n \times K \times K$ choices overall.
\end{proof}
\begin{comment}
\subsubsection{Generalization to sub-complexes}
From a theoretical perspective, results of this section generalize to a subcomplex
$\Delta' \subseteq \Delta(H_{\Gamma})$.  Then $\Gamma(v)$ will not include any vertices not 
adjacent to $v$ in $\Delta'$ and the definition of $\Gamma$-pertinent only applies to 
$2$-simplices actually present in $\Delta'$.  The ranking system on $\Delta'$ must still
be $3$-concordant, but only in the sense of being a global section of $\mcf_3$ 
(Definition \ref{d:careful-sheaf}).  The results and proofs in this section are still
valid in this greater generality.
\end{comment}

\subsection{In-sway}\label{s:insway}
The notion of \textit{in-sway} is intended to convey relative closeness of two
objects in $S$, with respect to a ranking system,
just as cohesion does for partitioned local depth \cite{ber}.
Computation of in-sway
is the driving principle of rank-based linkage: see Algorithm \ref{a:rbl}.

If we were dealing with the whole ranking system, rather than
the $\Gamma$-neighbor graph, then
the in-sway of the unordered pair $\{x, z\} \in \binom{S}{2}$ would be
the number of voter triangles $\{xy, yz, xz\}$ in which $xz$ is the source.
In the context of $\Gamma$-neighbors, the in-sway does not take account 
of all voter triangles, but only of the $\Gamma$-pertinent 2-simplices of Definition \ref{d:strictkpert}.
The principle will be: {\em the 2-simplices whose source 0-simplex receives in-sway are exactly the voter triangles 
bounded by 1-simplices which receive an orientation.}

\begin{defn}\label{d:ksway}
Given a 3-concordant ranking system on $S$, an integer $K < n$,
and a pair of mutual friends $x$ and $z$ in $S$,
the in-sway $\sigma(\{x,z\}):=\sigma_K(\{x,  z\})$ is the number of oriented $\Gamma$-pertinent 2-simplices
$\{xy, yz, xz\}$ in which $xz$ is the source.
\end{defn}

Why restrict the definition to mutual friends? 
Lemma \ref{l:ksway} shows that without mutual friendship the in-sway is zero, because such
 a pair $\{x, z\}$ could not be a source in a $\Gamma$-pertinent 2-simplex.

\begin{lem}\label{l:ksway}
Given an oriented $\Gamma$-pertinent 2-simplex $\{xy, yz, xz\}$
which is not a directed 3-cycle, conditions (a)
and (b) are equivalent.
\begin{enumerate}
    \item [(a)]
    $xz$ is the source in the 2-simplex.
    \item [(b)]
    $x$ and $z$ are mutual friends (i.e. $x \in \Gamma(z)$
    and $z \in \Gamma(x)$) and both i. and ii. hold:
    \begin{enumerate}
    \item[i.] If $y \in \Gamma(x)$, then $z \prec_x y$.
    \item[ii.] If $y \in \Gamma(z)$, then $x \prec_z y$.
\end{enumerate}
\end{enumerate}
\end{lem}

\begin{proof}
First assume (b). Condition i. ensures the direction 
 $xz \rightarrow xy$, while condition ii. ensures the direction 
 $xz \rightarrow yz$. Thus $xz$ is the source in the 2-simplex,
 which proves (a).
 
 Conversely assume (a), which gives the directions 
  $xz \rightarrow xy$ and  $xz \rightarrow yz$.
  The former implies that either  $\{y, z\} \subset \Gamma(x)$
  with $z \prec_x y$, or else
  \(
z \in \Gamma(x) \land y \notin \Gamma(x),
  \)
  by (\ref{e:1-simplices}).  The latter implies
  that either  $\{x, y\} \subset \Gamma(z)$
  with $x \prec_z y$, or else
  \(
x \in \Gamma(z) \land y \notin \Gamma(z).
  \)
Combining these shows that $x$ and $z$ are mutual friends
and i. and ii. hold, establishing (b).
\end{proof}

%%%%%%%%%%%%%%%%%%%%%%%%%%%%%%%%%%%%%%%%%%%%%%%%%%%%%%%%%%%%%%%%

\subsection{Formal definition of rank-based linkage}

Assume for Definitions \ref{d:linkage} and \ref{d:rbl} that we
are given a representative out-ordered digraph, as in Definition \ref{d:riwdg}.

\begin{defn}\label{d:linkage}
The linkage graph is the undirected edge-weighted graph 
$(S, \mathcal{L}, \sigma_K)$ 
whose links $\mathcal{L}$ are the pairs of mutual friends, and whose weight function $\sigma_K: \mathcal{L} \to \Z_+$
is the in-sway of Definition \ref{d:ksway}.
\end{defn}

\begin{defn}\label{d:rbl}
 Rank-based linkage with a cut-off at $t$ partitions $S$ into
 graph components of $(S, \mathcal{L}_t)$, 
 where $\mathcal{L}_t$ is defined to be the set of links in $\mathcal{L}$ whose in-sway is at least $t$.
This collection of partitions, which is refined for increasing
$t =  1, 2, \ldots$,
is the hierarchical partitioning scheme of rank-based linkage.
\end{defn}

\begin{prop}\label{p:rblcorrectness}
Algorithm \ref{a:rbl} (rank-based linkage) computes the
in-sway $\sigma_K(\{x,  z\})$ for all pairs of mutual friends $x$ and $z$
in no more than $O(n K^2)$ steps.
\end{prop}

\begin{proof}
To perform the rank-based linkage computation, it suffices to visit every
$\Gamma$-pertinent 2-simplex and determine which 0-simplex is the source.
By Lemma \ref{l:kpert}, every $\Gamma$-pertinent 2-simplex contains a pair of
mutual friends, say $\{x, z\}$. Hence it suffices to visit every
pair of mutual friends, and for such a pair to enumerate the 
$\Gamma$-pertinent 2-simplices $xyz$ in which $xz$ is the source.

The first part of Algorithm \ref{a:rbl} builds the set $\mathcal{L}$
of mutual friend pairs. If $\{x, z\}$ is such a pair, then
it suffices by the definition of the simplicial complex
to consider only those $y \notin \{x, z\}$ such that
at least one of $\{x, y\}$ and $\{z, y\}$
is in $\mathcal{U}_{\Gamma}$. By Definition \ref{d:strictkpert}, we must
also restrict to $y$ such that $\{x, z\} \cap \Gamma(y)$ is non-empty.
These conditions appear in Algorithm \ref{a:rbl}.

Given a triple $\{x, y, z\}$, Lemma \ref{l:ksway} (b) gives the conditions
for $xz$ to be the source in the 2-simplex $xyz$. These conditions
are reproduced in the predicate \textsc{FirstElementIsSource} of 
Algorithm \ref{a:rbl}.

The $O(n K^2)$  complexity assertion follows from Proposition \ref{p:kpertbound}.
\end{proof}

\subsection{Partitioned local depth versus rank-based linkage}\label{s:paldvsrbl}

Rank-based linkage was inspired partly by partitioned local depth (PaLD)
\cite{ber}. Figure \ref{f:pald10pts} illustrates the cluster graph produced by PaLD, run against the same
data from Table \ref{t:3conc10points} as was used to illustrate rank-based linkage in Figure \ref{f:rbl10pts}. In both cases, objects 0 and 6
form the most bonded pair, but objects 4 and 8 are not as strongly
linked by PaLD as by rank-based linkage.
To explain how the two methods differ, start by reframing
the notion of \textit{conflict focus} arising in the PaLD algorithm.

\begin{defn}\label{d:stalk}
Given an out-ordered digraph on $S$, and 
undirected $K$-NN graph $(S,\mathcal{U}_{\Gamma} )$, 
 \begin{equation}\label{e:sizestalk}
     \tau_K: \mathcal{U}_{\Gamma} \to \Z_+
 \end{equation} 
supplies the number $\tau_K(\{x, y\})$ of distinct $z$ for which the 
2-simplex $\{xy, yz, xz\}$
is $\Gamma$-pertinent, and $xy$ is not its source. 
\end{defn}
In other words $\tau_K(\{x, y\}) + \sigma_K(\{x,z\})$ counts the
$\Gamma$-pertinent 2-simplices in which $xy$ is a 0-face, for each
$\{x, y\} \in \mathcal{U}_{\Gamma}$, where we take the in-sway $\sigma_K(\{x,z\})$ to be zero when
$x$ and $t$ are not mutual friends.

\begin{figure}
\caption{\textbf{Partitioned local depth on Table \ref{t:3conc10points}:}
\textit{Six edges, two clusters, three outliers. Cohesion shown as edge weight.
Compare with Figure \ref{f:rbl10pts}}.
}
\label{f:pald10pts}
\begin{center}
\scalebox{0.5}{\includegraphics{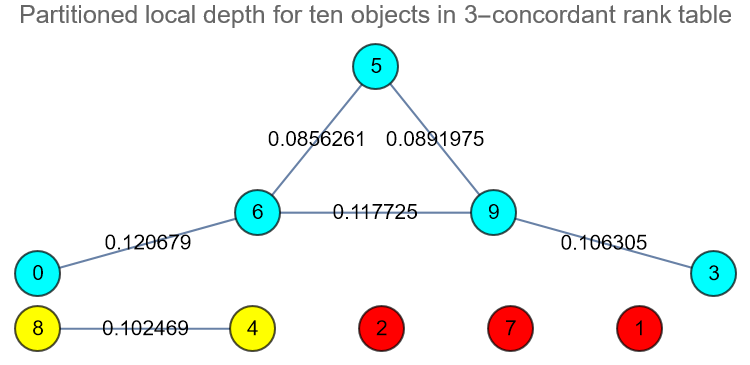} }
\end{center}
\end{figure}

\subsubsection{PaLD calculations}
The essence of PaLD \cite{ber} is to construct an asymmetric \textit{cohesion} $C_{x, z}$ as the sum
of the reciprocals of the sizes $2 + \tau_n(x,y)$ (here $K=n$)
 over $y \neq x$ such that $xz$ is the source in $\{xy, yz, xz\}$.
 The integer 2 refers to the inclusion of $x$ and $y$ in the count.
 When $x$ and $y$ are ``close'', meaning that both $r_x(y)$ and $r_y(x)$ are small,
 then $1/(2 + \tau_n(x,y))$ is relatively large,
 and therefore such a $y$ exerts more influence over the cohesion.
 The PaLD criterion for creating a link $\{x, z\}$ is
 that both $C_{x, z}$ and $C_{z, x}$ exceed a threshold defined 
 in terms of the trace of the cohesion matrix;
we omit details.

\subsubsection{Partitioned NN local depth}
 In practical contexts, computing (\ref{e:sizestalk}) entails cubic complexity in $n$. For large $n$ numerical rounding errors alone make cohesion
 hard to compute, regardless of run time.
 To mitigate this complexity, \cite{pannld} proposed an approximation
 to cohesion using $\Gamma$-neighbors, called PaNNLD, which involves
 randomization of comparisons missing from the out-ordered digraph.
 Complexity is bounded by the sum of squares of vertex degrees in the
 undirected $K$-NN graph.

 \subsubsection{Rank-based linkage}
Rank-based linkage takes a different approach,
  motivated not only by complexity considerations,
  but more importantly by category theory: see Theorem \ref{th:summary}.
The calculation in Algorithm \ref{a:rbl} uses uniform weighting, while taking $K \ll n$ so
that only $y$ adjacent to $x$ in the undirected $K$-NN graph $(S,\mathcal{U}_{\Gamma} )$
participate in the sum. This gives $O(n K^2)$ complexity.

 \subsubsection{Rank-based linkage with weighting}
One could compute the values (\ref{e:sizestalk}) (at least on 0-cells
of $\Gamma$-pertinent 2-simplices)
by inserting a couple of lines in Algorithm \ref{a:rbl},
assuming that $\tau_K$ has been initialized at zero for all pairs 
in $\mathcal{U}_{\Gamma}$.
If the \textsc{FirstElementIsSource} predicate in Algorithm \ref{a:rbl} returns true,
then increment both $\tau_K(\{x, y\})$ and $\tau_K(\{y,z\})$ by 1 (since neither $xy$ nor
$yz$ is the source of $\{xy, yz, xz\}$).
At the conclusion of Algorithm \ref{a:rbl}, both the
functions $\tau_K$ and $\sigma_K$ have been sufficiently evaluated.

Then there are two optional heuristics:
\begin{enumerate}
    \item
    Return $\sigma_K(\{x,y\})/(\sigma_K(\{x,y\}) + \tau_K(\{x, y\}))$, i.e.
the proportion of $\Gamma$-pertinent 2-simplices containing $xy$ in which $xy$ is the source, as a measure of linkage.
    \item
Alternatively build a weighted, but still symmetric, 
version of in-sway: run Algorithm \ref{a:rbl} again, but instead of incrementing $\sigma_K(\{x,z\})$ by 1
in the inner loop,
increment it by a weighted quantity such as
$1/ (2 + \min\{\tau_K(\{x, y\}), \tau_K(\{y,z\}) \})$ if 
\textsc{FirstElementIsSource} returns true. The effect will be that when $y$ is ``close'' to $x$ or $z$,
then the 2-simplex $\{xy, yz, xz\}$ exerts more influence on $\sigma_K(\{x,z\})$.
\end{enumerate}

These weighted versions are not expected to satisfy Theorem \ref{th:summary},
but are worth experimentation as a compromise between rank-based linkage and PaLD,
with $O(n K^2)$ complexity.

%%%%%%%%%%%%%%%%%%%%%%%%%%%%%%%%%%%%%%%%%%%%%%%%%%%%%%%%%%%%%%%%%%%
\section{Functorial properties of rank-based linkage}\label{s:functor}
\subsection{Overview}
In this section, we present some functorial properties of the rank-based linkage process.
The process from Section \ref{s:rbl-steps} can be broken into three stages, two of which are readily available in the literature.
The first and most important step of the rank-based linkage process takes an out-ordered digraph and produces an edge-weighted graph.
Next, the edge-weighted graph is pruned, subject to the weights of its edges.
Lastly, the pruned graph is used to create a partitioned set.

Section \ref{s:functormotivation} provides practical motivation for the rank-based linkage algorithm.
Section \ref{s:definitions} briefly discusses some constructions for digraphs, graphs, and edge-weighted graphs, which will be instrumental in building the rank-based linkage functor.
Section \ref{s:catoodigr} constructs the category of out-ordered digraphs, relating it back to traditional digraphs,
and Section \ref{s:linkagefunctor} constructs the key functor to edge-weighted graphs.
Section \ref{s:composition} combines the steps together into the functor implementing ranked-based linkage from out-ordered digraphs to partitioned sets.
Finally, Section \ref{s:applications} discusses the practical application of the rank-based linkage functor.
Unlike Section \ref{s:rblformal}, which assumes 3-concordance
 to simplify matters, calculations here apply to all out-ordered digraphs.

\subsection{Practical motivation}\label{s:functormotivation}
Suppose a data scientist applies unsupervised learning 
to a collection of objects $S$, with a specific comparator.
Clusters are created based on a $\Gamma$-neighbor graph on $S$.
Later further objects are added to the collection, giving a larger collection $S'$,
still using the same comparator. 

When unsupervised learning is applied
a second time to $S'$, it is not desirable for the original set of clusters
to ``break apart" (see Section \ref{s:partset} for precise language);
better for the elements of $S' \setminus S$ either to join existing
clusters, or else form new clusters. To achieve this, we 
show that rank-based linkage is a functor from
a category of out-ordered digraphs (Section
 \ref{s:catoodigr}) to a category of partitioned sets 
 (Section \ref{s:partset}). This desirable property is typically false for
 optimization-based clustering methods especially those which 
 specify the number of clusters in advance, such as $k$-means and non-negative matrix
 factorization.

Our presentation is modelled on 
Carlsson and Mémoli \cite{car}'s functorial approach to clustering
in the context of categories of finite metric spaces.
Indeed the rank-based linkage algorithm arose as a consequence
of the desire to emulate the results of \cite{car}.
The central technical issue is the definition of appropriate categories.
%We omit the proof that morphisms in these categories have associative composition.

\subsection{Definitions}\label{s:definitions}

In this section, we summarize some functorial relationships between categories of types of graphs, as  illustrated in Figure \ref{f:weightedgraphs}.
These relationships will be instrumental in proving that the process of rank-based linkage itself is also a functor from an appropriately constructed category.  Some of these results are well-known in the literature or are folklore results that can be proven without undue difficulty.  As such, some details will be omitted for brevity.  The intrepid reader is invited to investigate further in \cite{dorfler1980,grilliette9,grilliette6,hell1979,hell-nesetril,kilp2001}.

Section \ref{s:digraphs} briefly discusses the connections between graphs and digraphs, emphasizing the symmetric closure and its dual, the symmetric interior, both of which will be used in the construction of the rank-based linkage functor in Section \ref{s:linkagefunctor}.
Section \ref{s:partset} considers the category of partitioned sets, which is presented differently, but equivalently, to the category of \cite{car}.
Section \ref{s:wgraphs} concludes by introducing a category of edge-weighted graphs with edge-expansive maps, which will be the target of the rank-based linkage functor in Section \ref{s:linkagefunctor}.

Throughout this discussion, 1-edges and directed loops will be allowed in graphs and digraphs, respectively.

\subsubsection{Associations Between Graphs \& Digraphs}\label{s:digraphs}
\begin{figure}
$\xymatrix{
\Digra\ar@/_1.5pc/[dr]_{N^{\diamond}}\ar@/^1.5pc/[dr]^{N^{\star}}  &   &   &   \WGra\ar@/^/[dl]^{W_{t}^{\star}}\\
&   \SDigra\ar[ul]_{N}\ar@/_/[d]_{M^{\diamond}} &   \Gra\ar[l]_{Z}^{\cong}\ar@/^/[ur]^{W_{t}}\\
&   \Part\ar@/_/[u]_{M}
}$
\caption{ \textit{Graphs \& weighted Graphs}}\label{f:weightedgraphs}
\end{figure}

To set notation, let $\vec{V}(Q)$ and $\vec{E}(Q)$ be the vertex and edge sets, respectively, of a digraph $Q$.  Symmetric digraphs are a well-understood subclass of digraphs, which are deeply connected to (undirected) graphs.  We summarize these connections here, and refer the intrepid reader to \cite{grilliette9} for a more detailed exploration.  Formally, we take the following definition and notation.

\begin{defn}[{Symmetric digraph, \cite[p.\ 1]{hell-nesetril}}]
Let $\Digra$ be the category of digraphs with digraph homomorphisms.
A digraph $Q$ is \emph{symmetric} if $(x,y)\in\vec{E}(Q)$ implies $(y,x)\in\vec{E}(Q)$.  Let $\SDigra$ be the full subcategory of $\Digra$ consisting of all symmetric digraphs, and let $\xymatrix{\SDigra\ar[r]^{N} & \Digra}$ be the inclusion functor.
\end{defn}

Recall that a subcategory is \emph{reflective} if it is full and replete, and the inclusion functor admits a left adjoint \cite[Definition 3.5.2]{borceux1}.
Dually, a subcategory is \emph{coreflective} if it is full and replete, and the inclusion functor admits a right adjoint \cite[p.\ 119]{borceux1}.
The category $\SDigra$ is a full subcategory of $\Digra$ by definition, and repleteness is routine to demonstrate, so all that remains to consider is the existence of a left or a right adjoint.
Recall that the \emph{symmetric closure} of a binary relation $R$ is the smallest symmetric relation containing $R$, and the \emph{symmetric interior} of $R$ is the largest symmetric relation contained within $R$ \cite[p.\ 305]{floudas}.
Below, both constructions are applied to the edge set of a digraph, which is merely a binary relation on the vertex set.

\begin{defn}[{Symmetric closure, \cite[Definition 5.3]{grilliette9}}]\label{d:symclos}
For a digraph $Q$, define the symmetric digraph $N^{\diamond}(Q)$ by
\begin{itemize}
\item $\vec{V}N^{\diamond}(Q):=\vec{V}(Q)$,
\item $\vec{E}N^{\diamond}(Q):=\left\{(x,y):(x,y)\in\vec{E}(Q)\textrm{ or }(y,x)\in\vec{E}(Q)\right\}$.
\end{itemize}
\end{defn}

\begin{defn}[{Symmetric interior, \cite[Definition 5.5]{grilliette9}}]
For a digraph $Q$, define the symmetric digraph $N^{\star}(Q)$ by
\begin{itemize}
\item $\vec{V}N^{\star}(Q):=\vec{V}(Q)$,
\item $\vec{E}N^{\star}(Q):=\left\{(x,y):(x,y)\in\vec{E}(Q)\textrm{ and }(y,x)\in\vec{E}(Q)\right\}$.
\end{itemize}
\end{defn}

These two constructions correspond to the left and right adjoints of $N$, respectively.
The action of both adjoint functors is trivial on morphisms, which can be shown by means of the universal property of each.

\begin{comment}
\adam{I think these words should be defined $dots$ I think these functors should be defined before this statement is made}
\textcolor{red}{I agree with Adam. No idea what terms mean!}
\will{How does the exposition above read?} \adam{Much better, thanks} \textcolor{red}{I agree}
\end{comment}

For a graph $G$, let $V(G)$ and $E(G)$ be its vertex and edge sets, respectively.  An undirected graph can be thought of as a symmetric digraph, where each 2-edge is replaced with a directed 2-cycle and each 1-edge is replaced with a directed loop.  This natural association is formalized in an invertible functor between the category of graphs and the category of symmetric digraphs.

\begin{defn}[{Equivalent symmetric digraph, \cite[Definition 5.7]{grilliette9}}]\label{d:defineZ}
Let $\Gra$ be the category of graphs. For $\xymatrix{G\ar[r]^{f} & G'}\in\Gra$, define $Z:\Gra\to\SDigra$ by
\begin{itemize}
\item $\vec{V}Z(G):=V(G)$,
\item $\vec{E}Z(G):=\left\{(x,y):\{x,y\}\in E(G)\right\}$,
\item $Z(f):=f$.
\end{itemize}
\end{defn}

\subsubsection{Partitioned Sets}\label{s:partset}

Observe that a digraph is merely a ground set equipped with a binary relation upon that ground set.  When that relation is an equivalence relation, which partitions the ground set, we take the following definition.

\begin{defn}[{Partitioned set, \cite[p.\ 2]{larusson2006}}]
A digraph $X$ is a \emph{partitioned set} if ${\sim}_{X}:=\vec{E}(X)$ is an equivalence relation on $\vec{V}(X)$.  Consequently, a partitioned set is a type of symmetric digraph.  Let $\Part$ be the full subcategory of $\SDigra$ consisting of all partitioned sets, and let $\xymatrix{\Part\ar[r]^{M} & \SDigra}$ be the inclusion functor.
\end{defn}

Realizing partitioned sets in this way shows that $\Part$ is a reflective subcategory of $\SDigra$, where the reflector is the reflexive-transitive closure \cite[p.\ 3:2]{pous2018} of the edge set.  The equivalence relation generated in this process is the relation of connected components, where two vertices are related when there is a path from one to the other.  The proof of the universal property is analogous to that of the symmetric closure and will be omitted for brevity.

\begin{defn}[Reflexive-transitive closure]\label{d:rtclos}
For a symmetric digraph $Q$, define the partitioned set $M^{\diamond}(Q)$ by
\begin{itemize}
\item $\vec{V}M^{\diamond}(Q):=\vec{V}(Q)$,
\item $\vec{E}M^{\diamond}(Q):=\left\{
(x,y):
\exists\left(x_{k}\right)_{k=0}^{n}\left(\begin{array}{c}
x_{0}=x,\\
x_{n}=y,\\
\forall1\leq k\leq n\left(
\left(x_{k-1},x_{k}\right)\in\vec{E}(Q)
\vee
x_{k-1}=x_{k}
\right)
\end{array}\right)
\right\}
$.
\end{itemize}
Let $\xymatrix{Q\ar[r]^(0.4){\eta_{Q}} & MM^{\diamond}(Q)}\in\SDigra$ be the inclusion map.
\end{defn}

\begin{theorem}[Universal property]\label{d:reflexive-transitive}
If $\xymatrix{Q\ar[r]^(0.4){f} & M(X)}\in\SDigra$,
there is a unique
\[
\xymatrix{M^{\diamond}(Q)\ar[r]^(0.6){\hat{f}} & X}\in\Part
\]
 such that $M\left(\hat{f}\right)\circ\eta_{Q}=f$.
\end{theorem}

As with $N^{\diamond}$ and $N^{\star}$, the action of $M^{\diamond}$ on morphisms can be shown to be trivial by means of the universal property above.  Effectively, the left adjoint shows that digraph homomorphisms, and graph homomorphisms by means of the isomorphism $Z$, map connected components into connected components.  That is, the connected components of the domain refine the components induced by the codomain.
\begin{comment}
  \textcolor{red}{Could you give a one sentence interpretation? Morphisms of symmetric digraphs fit in with morphisms of partitioned sets?}
\adam{Maybe something like "morphisms of (symmetric) digraphs take vertices in the same connected component of the source to vertices in the same connected component of the target"?}
\textcolor{red}{RD: Yes, I like Adam's version.}
\will{
How does the above exposition read?  Does it setup the next paragraph about \cite{car}?
}
\adam{Looks good to me.}
\end{comment}

In \cite{car}, partitioned sets were defined in terms of a partition on the ground set, and homomorphisms of partitioned sets were required to refine partitions through preimage.  This ``set system'' representation of partitioned sets is isomorphic to the ``digraph'' representation that we have given above.  The correspondence between equivalence relations and partitions is well-understood, so the only lingering question concerns the homomorphism conditions, which we shall show are equivalent.  To that end, the following notation is taken for a partitioned set $X$ and $x\in\vec{V}(X)$:
\begin{itemize}
\item $A_{x,X}:=\left\{y\in\vec{V}(X):x{\sim}_{X}y\right\}$ is the equivalence class of $x$ in $X$;
\item $\mathcal{C}(X):=\vec{V}(X)/{\sim}_{X}$ is the set of all equivalence classes of $X$.
\end{itemize}

Recall the following definition of refinement, which is integral to morphisms in the ``set system'' representation.

\begin{defn}[{Refinement, \cite[p.\ 1]{apostolico1987}}]
Let $S$ be a set and $\mathcal{A},\mathcal{B}\subseteq\mathcal{P}(S)$ be partitions of $S$.  The partition $\mathcal{A}$ \emph{refines} $\mathcal{B}$ if for all $A\in\mathcal{A}$, there is $B\in\mathcal{B}$ such that $A\subseteq B$.
\end{defn}

Thus, we now show the equivalence of the two conditions.

\begin{prop}[Alternate morphism condition]\label{p:alternatemorphism}
Let $X$ and $X'$ be partitioned sets and $f:\vec{V}(X)\to\vec{V}\left(X'\right)$.  Then $\xymatrix{X\ar[r]^{f} & X'}\in\Part$ if and only if $\mathcal{C}(X)$ refines the family $\left\{f^{-1}(B):B\in\mathcal{C}\left(X'\right)\right\}$.
\end{prop}

\begin{proof}

$(\Rightarrow)$
If $U\in\mathcal{C}(X)$, there is $x\in\vec{V}(X)$ such that $U=A_{x,X}$.
Let $B:=A_{f(x),X'}$.
If $y\in U$, then $x{\sim}_{X}y$, so $f(x){\sim}_{X'}f(y)$.
Hence, $f(y)\in B$, yielding $y\in f^{-1}(B)$.
Therefore, $U\subseteq f^{-1}(B)$.

$(\Leftarrow)$
Let $x,y\in\vec{V}(X)$ satisfy that $x{\sim}_{X}y$.
If $U:=A_{x,X}$, then $y\in U$.
There is $B\in\mathcal{C}\left(X'\right)$ such that $U\subseteq f^{-1}(B)$.
There is $z\in\vec{V}\left(X'\right)$ such that $B=A_{z,X'}$.
As $x,y\in f^{-1}(B)$, $f(x),f(y)\in B$,
so $f(x){\sim}_{X'}z{\sim}_{X'}f(y)$.
As ${\sim}_{X'}$ is transitive, $f(x){\sim}_{X'}f(y)$.

\end{proof}

While these two viewpoints are equivalent, we favor the ``digraph'' perspective due to the reflective relationship with symmetric digraphs presented in Theorem \ref{d:reflexive-transitive}.
To pass between graphs and ranking systems, we use edge-weighted graphs, as in the migration flows example of Section \ref{s:migration}.

\subsubsection{Edge-weighted Graphs}\label{s:wgraphs}
Edge-weighted graphs are not novel \cite[p.\ 48]{kruskal1956}, though we now consider a category of such objects.  
Formally, we take the following definitions for objects and morphisms.

\begin{defn}[Expansive homomorphism]
Let $G$ be a graph.  An \emph{edge-weight} on $G$ is a function $w:E(G)\to[0,\infty]$.  The pair $(G,w)$ is an \emph{edge-weighted graph}.
For edge-weighted graphs $(G,w)$ and $\left(G',w'\right)$, an \emph{edge-expansive homomorphism} from $(G,w)$ to $\left(G',w'\right)$ is a function $f:V(G)\to V\left(G'\right)$ satisfying the following conditions:
\begin{itemize}
\item (preserve adjacency) if $\{x,y\}\in E(G)$, then $\left\{f(x),f(y)\right\}\in E\left(G'\right)$;
\item (expansive) if $\{x,y\}\in E(G)$, then $w\left(\{x,y\}\right)\leq w'\left(\left\{f(x),f(y)\right\}\right)$.
\end{itemize}
Let $\WGra$ be the category of edge-weighted graphs with edge-expansive homomorphisms.
\end{defn}

Much like \cite[Definition 3.4.6]{grilliette5}, a traditional graph can be assigned a constant weight function, to which we give the following notation.

\begin{defn}[Constant-weight functor]
Fix $t\in[0,\infty]$.  For a graph $G$, define the edge-weighted graph $W_{t}(G):=\left(G,c_{G,t}\right)$, where $c_{G,t}\left(\{x,y\}\right):=t$.  For $\xymatrix{G\ar[r]^{f} & G'}\in\Gra$, define $W_{t}(f):=f$.  
\end{defn}
Routine calculations show that $W_{t}$ is a functor from $\Gra$ to $\WGra$.
While this functor is fully faithful, it admits a non-trivial right adjoint, where the edge set is trimmed only to those edges exceeding the fixed value $t$.  This is analogous to the unit-ball functor in functional analysis \cite{joyofcats,goodearl1990,grandis2007,pelletier1989,pelletier1993}, and the proof is nearly identical.

\begin{defn}[Cut-off operation]\label{d:cutoff}
Fix $t\in[0,\infty]$.  For an edge-weighted graph $(G,w)$, define the graph $W_{t}^{\star}(G,w)$ by
\begin{itemize}
\item $VW_{t}^{\star}(G,w):=V(G)$,
\item $EW_{t}^{\star}(G,w):=\left\{\{x,y\}\in E(G):w\left(\{x,y\}\right)\geq t\right\}$.
\end{itemize}
Let $\xymatrix{W_{t}W_{t}^{\star}(G)\ar[rr]^(0.6){\theta_{(G,w)}} & & G}\in\WGra$ be defined by $\theta_{(G,w)}(v):=v$.
\end{defn}

\begin{theorem}[Universal property]
Fix $t\in[0,\infty]$.  If $\xymatrix{W_{t}\left(G'\right)\ar[r]^(0.6){f} & G}\in\WGra$, there is a unique $\xymatrix{G'\ar[r]^(0.4){\hat{f}} & W_{t}^{\star}(G)}\in\Gra$ such that $\theta_{(G,w)}\circ W_{t}\left(\hat{f}\right)=f$.
\end{theorem}

Use of the universal property shows that $W_{t}^{\star}$ acts trivially on morphisms.  Moreover, $W_{\infty}$ is right adjoint to $W_{0}^{\star}$, which is the forgetful functor from $\WGra$ to $\Gra$.  
Using the abbreviation ``$L\dashv R$'' for ``$L$ is left adjoint to $R$'', write
 $W_{0}\dashv W_{0}^{\star}\dashv W_{\infty}\dashv W_{\infty}^{\star}$, by analogy to \cite[p.\ 176]{grandis2007}.

%\textcolor{red}{What is the $\dashv$ symbol?}
%\will{One would read ``$L\dashv R$'' as ``$L$ is left adjoint to $R$''.  This is the same symbology used in \cite[p.\ 176]{grandis2007} for his analogous result.  I was hoping to channel him here.}

\begin{theorem}[Secondary universal property]
If $\xymatrix{W_{0}^{\star}\left(G,w\right)\ar[r]^(0.6){f} & G}\in\Gra$, there is a unique $\xymatrix{(G,w)\ar[r]^(0.4){\hat{f}} & W_{\infty}(G)}\in\WGra$ such that $W_{0}^{\star}\left(\hat{f}\right)=f$.
\end{theorem}

Lastly, there is a natural inclusion between the different cutoff functors.  For an edge-weighted graph $(G,w)$ and $0\leq s\leq t\leq\infty$, notice that $EW_{t}^{\star}(G,w)\subseteq EW_{s}^{\star}(G,w)$.  Thus, $W_{t}^{\star}(G,w)$ is a subgraph of $W_{s}^{\star}(G,w)$, where potentially some edges have been removed.  However, more can be said, and to do so, we take the following notation.

\begin{defn}[Natural inclusion]
Say $0\leq s\leq t\leq\infty$.  For an edge-weighted graph $(G,w)$, let $\xymatrix{W_{t}^{\star}(G,w)\ar[rr]^{\rho^{s,t}_{(G,w)}} & & W_{s}^{\star}(G,w)}\in\Gra$ be the inclusion map, and let $\rho^{s,t}:=\left(\rho^{s,t}_{(G,w)}\right)_{(G,w)\in\ob(\WGra)}$.
\end{defn}

As $\rho^{s,t}_{(G,w)}$ is bijective as a function, it is a bimorphism in $\Gra$.  Moreover, a quick diagram chase with a morphism $\xymatrix{(G,w)\ar[r]^{f} & \left(G',w'\right)}\in\WGra$ demonstrates that $\rho^{s,t}$ is a natural transformation from $W_{t}^{\star}$ to $W_{s}^{\star}$.  Likewise, a check shows that taken as a family, $\left(\rho^{s,t}\right)_{s\leq t\in[0,\infty]}$ acts like an inverse system in the following way.

\begin{prop}[Inverse system]
If $0\leq s\leq t\leq\infty$,
then $\rho^{s,t}$ is a natural bimorphism from $W_{t}^{\star}$ to $W_{s}^{\star}$.
Moreover, if $0\leq u\leq s\leq t\leq\infty$, then
$\rho^{u,s}\circ\rho^{s,t}=\rho^{u,t}$ and $\rho^{t,t}=id_{W_{t}^{\star}}$.
\end{prop}

\subsection{Out-ordered Digraphs}\label{s:catoodigr}
Out-ordered digraphs have already been introduced in Definition \ref{d:outorddigr}: a \textit{total order} is placed on the out-neighbor set $\Gamma(x)$, for each 
vertex $x$. This class of graph-like structures will be used to codify the rank-based linkage process as a functor. 
Indeed both Algorithm \ref{a:rbl} and the functorial computations can be performed under the weaker assumption of a 
\textit{partial order} on each out-neighbor set, as might occur in practical application of rank-based linkage to cases of missing data, although 
incomparable out-neighbors $y, z \in \Gamma(x)$ deny orientation from the 1-simplex $\{xy, xz\}$.

\begin{defn}[Partially out-ordered digraph]\label{d:pood}
A \emph{partially out-ordered digraph} is a triple $(S,\Gamma,\preceq)$ satisfying the following conditions:
\begin{itemize}
\item $S$ is the \emph{vertex set};
\item $\Gamma:S\to\mathcal{P}(S)$ is the \emph{neighborhood function}, i.e.\ $(x,y)$ is an
edge if $y \in \Gamma(x)$;
\item $\preceq:S\to\mathcal{P}(S\times S)$ is the \emph{neighborhood order}, where $\preceq_{x}:=\preceq(x)$ is a partial order on $\Gamma(x)$ for all $x\in S$.
\end{itemize}
\end{defn}
To clarify the last bullet, $y \preceq_x z$ means that $(y,z)$ is in the image of $x$ under $\preceq$,
which is a set of ordered pairs of elements of $\Gamma(x)$.

Indeed, this structure is merely a digraph, where each out-neighborhood has been endowed with a partial order, ranking its out-neighbors.  One can always impose this structure on a directed graph in a trivial way using the discrete ordering.  However, the following example takes the form of the input to the rank-based linkage algorithm, as we saw in Section \ref{s:datainput}.

%\will{I reworded your transitionary sentence and integrated it into the preceding paragraph.}

%\begin{ex}[Digraph with discrete ordering]\label{e:trivialorder}
%Let $Q$ be a digraph.  Define a partially out-ordered digraph $(S,\Gamma,\preceq)$ by
%\begin{itemize}
%\item (vertex set) $S:=\vec{V}(Q)$;
%\item (out-neighborhood) $\Gamma(x):=\left\{y\in S:(x,y)\in\vec{E}(Q)\right\}$;
%\item (discrete order) for $x\in S$ and $y,z\in\Gamma(x)$, define $y\preceq_{x}z \iff y=z$.
%\end{itemize}
%\end{ex}

\begin{ex}[Edge-weighted digraph]\label{e:wodg}
Let $Q$ be a digraph and $w:\vec{E}(Q)\to[0,\infty]$ be a weight function on the edges.  Define a partially out-ordered digraph $(S,\Gamma,\preceq)$ by
\begin{itemize}
\item (vertex set) $S:=\vec{V}(Q)$;
\item (out-neighborhood) $\Gamma(x):=\left\{y\in S:(x,y)\in\vec{E}(Q)\right\}$;
\item (weight order) for $x\in S$ and $y,z\in\Gamma(x)$, define $y\preceq_{x}z$ iff one of the two conditions holds:
\begin{enumerate}
\item $y=z$,
\item $w(x,y)>w(x,z)$.
\end{enumerate}
\end{itemize}
\end{ex}

The homomorphisms between partially out-ordered digraphs have some natural qualities, mainly preserving the neighborhood and the order.  However, we enforce two other conditions, which will be needed in Theorem \ref{th:functor}.  First, the morphisms must be injections, potentially adding elements but never quotienting them together.  Second, any new elements that have been appended to a neighborhood must be ranked strictly farther than those already present.  Effectively, this latter condition means that a neighborhood in the codomain is an ordinal sum \cite[p.\ 547]{morgado1961} of the image of the neighborhood in the domain with some new elements.

\begin{defn}[Neighborhood-ordinal injection]\label{d:noddi}
Given two partially out-ordered digraphs $(S,\Gamma,\preceq)$ and $\left(S',\Gamma',\preceq'\right)$, a \emph{neighborhood-ordinal injection} is a function $\iota:S\to S'$ satisfying the following conditions:
\begin{itemize}
\item (neighborhood-preserving) for all $x,y\in S$, $y\in\Gamma(x)$ implies $\iota(y)\in\Gamma'\left(\iota(x)\right)$;
\item (order-preserving) for all $x\in S$ and $y,z\in\Gamma(x)$, $y\preceq_{x}z$ implies $\iota(y){\preceq'}_{\iota(x)}\iota(z)$;
\item (one-to-one) for all $x,y\in S$, $\iota(x)=\iota(y)$ implies $x=y$;
\item (ordinal sum)\footnote{
Notation: if $\xymatrix{S\ar[r]^{f} & S'}\in\mathbf{Set}$, 
then $\xymatrix{\mathcal{P}(S)\ar[r]^{\mathcal{P}(f)} & \mathcal{P}\left(S'\right)}\in\mathbf{Set}$ is determined by
\begin{center}$
\mathcal{P}(f)(A)=\left\{f(s):s\in A\right\}, \quad A\in\mathcal{P}(S).
$\end{center}
} for all $x\in S$ and $w\in\Gamma'\left(\iota(x)\right)$, 
$w\not\in\mathcal{P}(\iota)\left(\Gamma(x)\right)$

implies $\iota(y){\prec'}_{\iota(x)}w$ for all $y\in\Gamma(x)$.
\end{itemize}
\end{defn}

\noindent \textbf{Remark: }The rationale for Definition \ref{d:noddi} comes from the oriented 
simplicial complex perspective of Section \ref{s:scprt}.
Suppose there is a 2-simplex $\{xy, yz, xz\}$ whose boundary has an orientation given by $(S,\Gamma,\preceq)$,
making $xz$ the source 0-simplex. Injection from $(S,\Gamma,\preceq)$ into $\left(S',\Gamma',\preceq'\right)$
must preserve orientations. In particular 0-simplex $xz$ retains its status as source, which ensures that the in-sway $\sigma(x, z)$ 
cannot decrease. The effect is that $x$ and $z$ cannot become ``unglued'' from each other when rank-based linkage is applied to the range.

While identity maps satisfy these conditions trivially, composition of morphisms is not trivial due to the ordinal sum condition.

\begin{lem}[Composition]
Let $(S,\Gamma,\preceq)$, $\left(S',\Gamma',\preceq'\right)$, $\left(S'',\Gamma'',\preceq''\right)$ be partially out-ordered digraphs.
Say $\iota$ is a neighborhood-ordinal injection from $(S,\Gamma,\preceq)$ to $\left(S',\Gamma',\preceq'\right)$,
and $\kappa$ is a neighborhood-ordinal injection from $\left(S',\Gamma',\preceq'\right)$ to $\left(S'',\Gamma'',\preceq''\right)$.
Then $\kappa\circ\iota$ is a neighborhood-ordinal injection from $(S,\Gamma,\preceq)$ to $\left(S'',\Gamma'',\preceq''\right)$.
\end{lem}

\begin{proof}

It is routine to show that composition preserves neighborhoods, orders, and monomorphisms.
To show the ordinal sum condition,
let $x\in S$ and $w\in\Gamma''\left((\kappa\circ\iota)(x)\right)$ satisfy that $w\not\in\mathcal{P}(\kappa\circ\iota)\left(\Gamma(x)\right)$.
Note that
\begin{center}$
\mathcal{P}(\kappa\circ\iota)\left(\Gamma(x)\right)
\subseteq\mathcal{P}(\kappa)\left(\Gamma'\left(\iota(x)\right)\right)
\subseteq\Gamma''\left((\kappa\circ\iota)(x)\right).
$\end{center}
Hence, consider the following two cases.
\begin{enumerate}

\item Assume that $w\not\in\mathcal{P}(\kappa)\left(\Gamma'\left(\iota(x)\right)\right)$.
Then $\kappa(u)$ precedes $w$ according to the ranking at $\kappa\circ\iota(x)$ for all $u\in\Gamma'(\iota(x))$.
%Then, $\kappa(u){\prec''}_{(\kappa\circ\iota)(x)}w$ for all $u\in\Gamma'(\iota(x))$.
In particular if $y\in\Gamma(x)$, then $\iota(y)\in\Gamma'(\iota(x))$, so 
$\kappa\circ\iota(y)$ precedes $w$ according to $\kappa\circ\iota(x)$, written as
$\kappa\circ\iota(y){\prec''}_{\kappa\circ\iota(x)}w$.

\item Assume that $w\in\mathcal{P}(\kappa)\left(\Gamma'\left(\iota(x)\right)\right)$.
Then, there is $u\in\Gamma'(\iota(x))$ such that $\kappa(u)=w$.
If $u=\iota(v)$ for some $v\in\Gamma(x)$, then $w=(\kappa\circ\iota)(v)$,
contradicting that $w\not\in\mathcal{P}(\kappa\circ\iota)\left(\Gamma(x)\right)$.
Thus, $u\not\in\mathcal{P}(\iota)\left(\Gamma(x)\right)$.
Therefore, $\iota(y){\prec'}_{\iota(x)}u$ for all $y\in\Gamma(x)$,
meaning $(\kappa\circ\iota)(y){\preceq''}_{(\kappa\circ\iota)(x)}\kappa(u)=w$ for all $y\in\Gamma(x)$.
If $(\kappa\circ\iota)(y)=w$ for some $y\in\Gamma(x)$, then $w\in\mathcal{P}(\kappa\circ\iota)\left(\Gamma(x)\right)$,
contrary to the standing assumption.
Therefore, $(\kappa\circ\iota)(y){\prec''}_{(\kappa\circ\iota)(x)}w$.

\end{enumerate}

\end{proof}

Let $\ood$ be the category of partially out-ordered digraphs with neighborhood-ordinal injections.  Using a constant weight in Example \ref{e:wodg}, every digraph can be given a neighborhood order.  However, the reverse is also true, that the neighborhood order can be forgotten to return a traditional digraph.  This forgetful functor will be instrumental in the next section to build the linkage functor.  The proof of functoriality is routine.

\begin{defn}[Forgetful functor]
For a partially out-ordered digraph $(S,\Gamma,\preceq)$, define the digraph $F(S,\Gamma,\preceq)$ by
\begin{itemize}
\item $\vec{V}F(S,\Gamma,\preceq):=S$,
\item $\vec{E}F(S,\Gamma,\preceq):=\left\{(x,y):y\in\Gamma(x)\right\}$.
\end{itemize}
For $\xymatrix{(S,\Gamma,\preceq)\ar[r]^{\iota} & \left(S',\Gamma',\preceq'\right)}\in\ood$, define $F(\iota):=\iota$.
\end{defn}

\subsection{Linkage Functor}\label{s:linkagefunctor}

In this section, we construct the linkage functor, which takes  a partially out-ordered digraph and returns an edge-weighted graph.  Taking motivation from Section \ref{s:outorddigr}, we define different notions of ``friendship'', which will be used to count elements toward the in-sway of Algorithm \ref{a:rbl}.  Section \ref{s:friendship} provides the definitions formally and ties each to the different graph constructions from Section \ref{s:digraphs}.  Section \ref{s:construction} uses these notions to construct the linkage graph of Section \ref{s:rbl-steps}, and demonstrates that the construction is functorial.

\subsubsection{Friendship Notions}\label{s:friendship}

Following the intuition of Section \ref{s:outorddigr}, we take the following formal definitions of ``friend'' and ``mutual friends''.  We also make the definition of ``acquaintances'', which will be dual to and weaker than ``mutual friends''.

\begin{defn}[Friends]
Let $(S,\Gamma,\preceq)$ be a partially out-ordered digraph.  For $x,y\in S$, $y$ is a \emph{friend} of $x$ if $y\in\Gamma(x)$.  Likewise, $x$ and $y$ are \emph{mutual friends} if $y$ is a friend of $x$ and $x$ is a friend of $y$.
Dually, $x$ and $y$ are \emph{acquaintances} if either $y$ is a friend of $x$ or $x$ is a friend of $y$.
\end{defn}

Each of these friendship notions in $(S,\Gamma,\preceq)$ corresponds to a type of edge in $F(S,\Gamma,\preceq)$.  The proof of each characterization is routine and left to the reader. Recall the functor $Z$ from
Definition \ref{d:defineZ}.

\begin{lem}[Characterization of friends]\label{l:friendship}
Let $(S,\Gamma,\preceq)$ be a partially out-ordered digraph and $x,y\in S$.
\begin{enumerate}
\item $y$ is a friend of $x$ iff $(x,y)\in\vec{E}F(S,\Gamma,\preceq)$.
\item $x$ and $y$ are mutual friends iff $\{x,y\}\in EZ^{-1}N^{\star}F(S,\Gamma,\preceq)$.
\item $x$ and $y$ are acquaintances iff $\{x,y\}\in EZ^{-1}N^{\diamond}F(S,\Gamma,\preceq)$.
\end{enumerate}
\end{lem}

Carrying the relationship metaphor forward, one can consider a ``third-wheel'', someone who is looking at a relationship between two mutual friends from the outside.  This idea can be compared to a parasocial relationship of a viewer of two content streamers, who commonly promote one another's content.

To formalize this notion, let $x$ and $z$ be the mutual friends, and let $y$ be the outside observer.  The most natural situation would be for $x$ and $z$ to be friends of $y$, i.e.\ $y$ watches the content of both creators.  However, it is possible that $y$ is also a creator, so $x$ or $z$ could frequent the content of $y$.  Thus, we must also consider the cases when $y$ is a friend of $x$ or $z$, yielding the three cases listed in Figure \ref{f:parasocial}.  Such $y$ are the ``voters'' for in-sway from Section \ref{s:rbl-steps}, and we take the following definition as the set of all such voters.

\begin{figure}
\begin{center}$\begin{array}{ccc}
\xymatrix{
x\ar@/^/[rr]   &   &   z\ar@/^/[ll]\\
&   y\ar[ul]\ar[ur]
}
&
\xymatrix{
x\ar@/^/[rr]   &   &   z\ar@/^/[ll]\\
&   y\ar@{<-}[ul]\ar[ur]
}
&
\xymatrix{
x\ar@/^/[rr]   &   &   z\ar@/^/[ll]\\
&   y\ar[ul]\ar@{<-}[ur]
}
\end{array}$\end{center}
\caption{In-sway Cases}
\label{f:parasocial}
\end{figure}

\begin{defn}[Third-wheels]
Let $(S,\Gamma,\preceq)$ be a partially out-ordered digraph.  For $x,z\in S$, define the following set:
\begin{center}$\begin{array}{rcl}
\mathscr{F}_{S,\Gamma}(x,z)
&   :=  &   \left\{y\in S:x\in\Gamma(y),z\in\Gamma(y)\right\}\\
&   &   \cup\left\{y\in S:y\in\Gamma(x),z\in\Gamma(y)\right\}\\
&   &   \cup\left\{y\in S:x\in\Gamma(y),y\in\Gamma(z)\right\}.
\end{array}$\end{center}
\end{defn}

Please note that we purposefully omit the situation where $y$ is a friend of $x$ and $z$, but $x$ and $z$ are not friends of $y$.  Returning to the content creator analogy, this is the case where $y$ does not frequent the content of either $x$ or $z$, but they both frequent the content of $y$.  Such a $y$ would not vote, through methods such as likes or subscriptions, on the content of either $x$ or $z$.

Moreover, the conditions defining the set $\mathscr{F}_{(S,\Gamma)}(x,z)$ are equivalent to those used to iterate Algorithm \ref{a:rbl}, noting that the underlying graph of $(S,\Gamma,\preceq)$ is precisely the graph corresponding to the symmetric closure of $F(S,\Gamma,\preceq)$.

\begin{lem}[Alternate characterization]\label{l:alternatef}
Let $(S,\Gamma,\preceq)$ be a partially out-ordered digraph and fix $x,z\in S$.  For all $y\in S$, $y\in\mathscr{F}_{(S,\Gamma)}(x,z)$ if and only if the following conditions hold:
\begin{enumerate}
\item $\{x,z\}\cap\Gamma(y)\neq\emptyset$,
\item $\{x,y\}\in EZ^{-1}N^{\diamond}F(S,\Gamma,\preceq)$,
\item $\{z,y\}\in EZ^{-1}N^{\diamond}F(S,\Gamma,\preceq)$.
\end{enumerate}
Consequently, $\mathscr{F}_{(S,\Gamma)}(x,z)=\mathscr{F}_{(S,\Gamma)}(z,x)$.
\end{lem}

The proof of the lemma is tedious, but routine, and provides an alternate characterization of the set $\mathscr{F}_{(S,\Gamma)}(x,z)$:  all $y$ that are acquaintances of both $x$ and $z$, but have at least one of $x$ and $z$ as a friend.

\subsubsection{Construction}\label{s:construction}

At last, the in-sway for a pair of mutual friends can be computed as detailed in Algorithm \ref{a:rbl}, creating the linkage graph described in Section \ref{s:rbl-steps}.

\begin{defn}[In-sway weight function]\label{d:inswayfunctor}
Let $(S,\Gamma,\preceq)$ be a partially out-ordered digraph.  For mutual friends $x,z\in S$, define the following set:
\begin{center}$
\mathscr{G}_{(S,\Gamma,\preceq)}(x,z)
:=\left\{y\in\mathscr{F}_{S,\Gamma}(x,z):\begin{array}{c}
y\in\Gamma(x)\implies z\prec_{x}y,\\
y\in\Gamma(z)\implies x\prec_{z}y
\end{array}\right\}.
$\end{center}
Define $\sigma_{(S,\Gamma,\preceq)}:EZ^{-1}N^{\star}F(S,\Gamma,\preceq)\to[0,\infty]$ by $\sigma_{(S,\Gamma,\preceq)}\left(\{x,z\}\right):=\card\left(\mathscr{G}_{(S,\Gamma,\preceq)}(x,z)\right)$ and $\mathcal{L}(S,\Gamma,\preceq):=\left(Z^{-1}N^{\star}F(S,\Gamma,\preceq),\sigma_{(S,\Gamma,\preceq)}\right)$.  For $\xymatrix{(S,\Gamma,\preceq)\ar[r]^{\iota} & \left(S',\Gamma',\preceq'\right)}\in\ood$, define $\mathcal{L}(\iota):=\iota$.
\end{defn}

Our goal is to show that $\mathcal{L}$ defines a functor from $\ood$ to $\WGra$.  To that end, we prove a pair of technical lemmas, which will culminate in the functoriality of $\mathcal{L}$ and a commutativity in Figure \ref{f:constructiondiagram}.

\begin{figure}
$\xymatrix{
&   &   \ood\ar[dll]_{F}\ar[drr]^{\mathcal{L}}\\
\Digra\ar@/_1.5pc/[dr]_{N^{\diamond}}\ar@/^1.5pc/[dr]^{N^{\star}}  &   &   &   &    \WGra\ar@/^/[dl]^{W_{t}^{\star}}\\
&   \SDigra\ar[ul]_{N}\ar@/_/[d]_{M^{\diamond}} &   &   \Gra\ar[ll]_{Z}^{\cong}\ar@/^/[ur]^{W_{t}}\\
&   \Part\ar@/_/[u]_{M}
}$
\caption{Linkage Construction}
\label{f:constructiondiagram}
\end{figure}

\begin{lem}\label{l:technical1}
Let $\xymatrix{(S,\Gamma,\preceq)\ar[r]^{\iota} & \left(S',\Gamma',\preceq'\right)}\in\ood$, $x\in S$, and $z\in\Gamma(x)$.  If $y\in\Gamma(x)\implies z\prec_{x}y$ and $\iota(y)\in\Gamma'\left(\iota(x)\right)$, then $\iota(z){\prec'}_{\iota(x)}\iota(y)$.
\end{lem}

\begin{proof}

Consider the following two cases.
\begin{enumerate}

\item Assume $y\in\Gamma(x)$.
Then, $z\prec_{x}y$,
so $\iota(y),\iota(z)\in\Gamma'\left(\iota(x)\right)$ and $\iota(z){\preceq'}_{\iota(x)}\iota(y)$.
For purposes of contradiction, assume that $\iota(z)=\iota(y)$.
As $\iota$ is one-to-one, $z=y$,
which contradicts the strict inequality $z\prec_{x}y$.
Therefore, $\iota(z){\prec'}_{\iota(x)}\iota(y)$.

\item Assume $y\not\in\Gamma(x)$.
For purposes of contradiction, assume $\iota(y)\in\mathcal{P}(\iota)\left(\Gamma(x)\right)$.
Then there is $u\in\Gamma(x)$ such that $\iota(u)=\iota(y)$.
As $\iota$ is one-to-one, $u=y$,
contradicting that $y\not\in\Gamma(x)$.
Therefore, $\iota(y)\not\in\mathcal{P}(\iota)\left(\Gamma(x)\right)$,
so $\iota(z){\prec'}_{\iota(x)}\iota(y)$ by the ordinal sum condition on $\iota$.

\end{enumerate}

\end{proof}

\begin{lem}\label{l:technical2}
Let $\xymatrix{(S,\Gamma,\preceq)\ar[r]^{\iota} & \left(S',\Gamma',\preceq'\right)}\in\ood$ and $x,z\in S$.  If $x$ and $z$ are mutual friends, then $\mathcal{P}(\iota)\left(\mathscr{G}_{(S,\Gamma,\preceq)}(x,z)\right)\subseteq\mathscr{G}_{\left(S',\Gamma',\preceq'\right)}\left(\iota(x),\iota(z)\right)$.
\end{lem}

\begin{proof}

If $y\in\mathscr{G}_{(S,\Gamma,\preceq)}(x,z)$, then $y\in\mathscr{F}_{(S,\Gamma)}(x,z)$.
By Lemma \ref{l:alternatef},
\begin{enumerate}
\item $\{x,z\}\cap\Gamma(y)\neq\emptyset$,
\item $\{x,y\}\in EZ^{-1}N^{\diamond}F(S,\Gamma,\preceq)$,
\item $\{z,y\}\in EZ^{-1}N^{\diamond}F(S,\Gamma,\preceq)$.
\end{enumerate}
As $Z^{-1}N^{\diamond}F$ is a functor, $Z^{-1}N^{\diamond}F(\iota)=\iota$ is a graph homomorphism from $Z^{-1}N^{\diamond}F(S,\Gamma,\preceq)$ to $Z^{-1}N^{\diamond}F\left(S',\Gamma',\preceq'\right)$.
Thus, $\left\{\iota(x),\iota(y)\right\},\left\{\iota(z),\iota(y)\right\}\in EZ^{-1}N^{\diamond}F\left(S',\Gamma',\preceq'\right)$.
Also,
\begin{center}$\begin{array}{rcl}
\{x,z\}\cap\Gamma(y)\neq\emptyset
&   \iff    &   x\in\Gamma(y)\vee z\in\Gamma(y)\\
&   \implies    &   \iota(x)\in\Gamma'\left(\iota(y)\right)\vee\iota(z)\in\Gamma'\left(\iota(y)\right)\\
&   \iff    &   \left\{\iota(x),\iota(z)\right\}\cap\Gamma'\left(\iota(y)\right)\neq\emptyset.
\end{array}$\end{center}
Therefore, $\iota(y)\in\mathscr{F}_{\left(S',\Gamma'\right)}\left(\iota(x),\iota(z)\right))$.

As $y\in\mathscr{G}_{(S,\Gamma,\preceq)}(x,z)$, one has $y\in\Gamma(x)\implies z\prec_{x}y$ and $y\in\Gamma(z)\implies x\prec_{z}y$.
By Lemma \ref{l:technical1}, $\iota(y)\in\Gamma'\left(\iota(x)\right)\implies\iota(z){\prec'}_{\iota(x)}\iota(y)$ and $\iota(y)\in\Gamma'\left(\iota(z)\right)\implies\iota(x){\prec'}_{\iota(z)}\iota(y)$.
Hence, $\iota(y)\in\mathscr{G}_{\left(S',\Gamma',\preceq'\right)}\left(\iota(x),\iota(z)\right)$.

\end{proof}

\begin{theorem}[Functoriality]\label{th:functor}
The maps $\mathcal{L}$ define a functor from $\ood$ to $\WGra$.  Moreover, $ZW_{0}^{\star}\mathcal{L}=N^{\star}F$.
\end{theorem}

\begin{proof}

Let $\xymatrix{(S,\Gamma,\preceq)\ar[r]^{\iota} & \left(S',\Gamma',\preceq'\right)}\in\ood$.
As $Z^{-1}N^{\star}F$ is a functor, $Z^{-1}N^{\star}F(\iota)=\iota$ is a graph homomorphism from $Z^{-1}N^{\star}F(S,\Gamma,\preceq)$ to $Z^{-1}N^{\star}F\left(S',\Gamma',\preceq'\right)$.
If $\{x,z\}\in EZ^{-1}N^{\star}F(S,\Gamma,\preceq)$, then $x$ and $z$ are mutual friends by Lemma \ref{l:friendship}.
By Lemma \ref{l:technical2},
\begin{center}$\begin{array}{rcl}
\sigma_{(S,\Gamma,\preceq)}\left(\{x,z\}\right)
&   =   &   \card\left(\mathscr{G}_{(S,\Gamma,\preceq)}(x,z)\right)\\
&   =   &   \card\left(\mathcal{P}(\iota)\left(\mathscr{G}_{(S,\Gamma,\preceq)}(x,z)\right)\right)\\
&   \leq   &   \card\left(\mathscr{G}_{\left(S',\Gamma',\preceq'\right)}\left(\iota(x),\iota(z)\right)\right)\\
&   =   &   \sigma_{\left(S',\Gamma',\preceq'\right)}\left(\left\{\iota(x),\iota(z)\right\}\right).
\end{array}$\end{center}
Thus, $\mathcal{L}(\iota)=\iota$ is an edge-expansive homomorphism from $\mathcal{L}(S,\Gamma,\preceq)$ to $\mathcal{L}\left(S',\Gamma',\preceq'\right)$.  Moreover,
\begin{itemize}
\item $
ZW_{0}^{\star}\mathcal{L}(S,\Gamma,\preceq)
=ZW_{0}^{\star}\left(Z^{-1}N^{\star}F(S,\Gamma,\preceq),\sigma_{(S,\Gamma,\preceq)}\right)
=ZZ^{-1}N^{\star}F(S,\Gamma,\preceq)$

$={N^{\star}F(S,\Gamma,\preceq)}
$,
\item $
ZW_{0}^{\star}\mathcal{L}(\iota)
=ZW_{0}^{\star}(\iota)
=Z(\iota)
=\iota
=N^{\star}(\iota)
=N^{\star}F(\iota)
$.
\end{itemize}
Routine calculations show that $\mathcal{L}$ preserves compositions and identities.

\end{proof}

\subsection{Conclusion: the rank-based linkage functor}\label{s:composition}
Finally, we assemble the functors from Figure \ref{f:constructiondiagram} into a composite functor, which has the action of rank-based linkage.

\begin{defn}[Rank-based linkage functor]
Fix $t\in[0,\infty]$.
Define $\mathcal{R}_{t}:=M^{\diamond}ZW_{t}^{\star}\mathcal{L}$.
For $0\leq s\leq t\leq\infty$, let $\nu^{s,t}:=id_{M^{\diamond}Z}\ast\rho^{s,t}\ast id_{\mathcal{L}}$, the horizontal product\footnote{
The \emph{horizontal composition} \cite[p.\ 46]{leinster} (i.e.\ \emph{star product} \cite[Exercise 6A.a]{joyofcats} or \emph{Godement product} \cite[p.\ 13]{borceux1}) of natural transformations $\alpha:F\to G$ and $\beta:H\to K$ is determined by $\left(\beta\ast\alpha\right)_{A}:=\beta_{G(A)}\circ H\left(\alpha_{A}\right)=K\left(\alpha_{A}\right)\circ\beta_{F(A)}$.
} of $\rho^{s,t}$ with the respective identity transformations.
\end{defn}

\begin{figure}
\begin{center}$\xymatrix{
\ood\ar[rr]^{\mathcal{L}} & &
\WGra\rrtwocell^{W_{t}^{\star}}_{W_{s}^{\star}}{\;\;\;\;\rho^{s,t}} & &
\Gra\ar[rr]^{M^{\diamond}Z} & &
\Part
}$\end{center}
\caption{Linkage Functor \& Transformation}
\end{figure}

The functor $\mathcal{R}_{t}$ implements the rank-based linkage algorithm from Section \ref{s:rbl-steps} precisely:
\begin{enumerate}
\item $\mathcal{L}$ (Definition \ref{d:inswayfunctor}) computes the weighted linkage graph;
\item $W_{t}^{\star}$ (Definition \ref{d:cutoff}) prunes all edges with in-sway below $t$;
\item $M^{\diamond}Z$ (Definitions \ref{d:defineZ}, \ref{d:rtclos}) determines connected components of the resulting graph.
\end{enumerate}
Moreover, observe that by Theorem \ref{th:functor},
$\mathcal{R}_{0}=M^{\diamond}ZW_{0}^{\star}\mathcal{L}=M^{\diamond}N^{\star}F$,
which returns the connected components of the linkage graph with no pruning of edges.
Furthermore, the natural transformation $\nu^{s,t}$ passes from $\mathcal{R}_{t}$ to $\mathcal{R}_{s}$, preventing the resulting partitions from breaking apart as the values of $t$ decrease.
The work of this section may be concisely expressed as:

\begin{theorem}[SUMMARY]\label{th:summary}
The rank-based linkage algorithm is the functor $\mathcal{R}_{t}$ from the category of partially out-ordered digraphs
to the category of partitioned sets.
\end{theorem}

\subsection{Data science application}\label{s:applications}
Return to the context of Section \ref{s:functormotivation}, where
the scientist wishes to perform rank-based linkage
on a superset $S' \supset S$ of the set to which it was originally applied.
How do \textit{neighborhood-ordinal injections} (Definition \ref{d:noddi}) -- the
morphisms in the category of partially out-ordered digraphs -- appear in this context?

Suppose the partially out-ordered digraph on objects $S$ was built on the $K$-NN digraph.
The partially out-ordered digraph on objects $S' \subset S$ must be built on the $K'$-NN digraph
where $K' \geq K$ is large enough so that if $y$ is one of 
the $K$-NN of $z$ in $(S, \{\prec_x, x \in S\})$,  then it must be one of
the $K'$-NN of $z$ in $(S',  \{\prec'_x, x \in S\})$. This will ensure that
the injection  $S \to S'$ is a neighborhood-ordinal injection, i.e. a morphism in
the category of partially out-ordered digraphs. 

Theorem \ref{th:functor} assures us that
a pair $\{x, z\}$ of mutual friends in $(S, \Gamma , \preceq)$ whose in-sway is at least $t$
will also have an insway at least $t$ when rank-based linkage is re-applied to $(S', \Gamma' , \preceq')$.
This ensures that if objects $x$ and $y$ are in the same component of $(S, \mathcal{L}_t)$,
then they are also in the same component of $(S', \mathcal{L}'_t)$. This is the functorial
property of rank-based linkage.

%%%%%%%%%%%%%%%%%%%%%%%%%%%%%%%%%%%%%%%%%%%%%%%%%%%%%%%%%%%%
\section{Data merge: gluing ranking systems together}\label{s:gluers}

\subsection{Data science context: merging sets of data objects}

Principled methods of merging methods and data of unsupervised learning projects
are rare, as the survey \cite{lah} shows. 
A categorical approach to merging databases is presented by
Brown, Spivak \& Wisnesky \cite{bro}, but here we consider the unsupervised learning issue.
Proposition \ref{p:sheafify} and Theorem \ref{th:summary} have direct application to 
pooling ranking systems on sets of data objects.

\noindent \textbf{Genetics: } Alice and Bob pursue bio-informatics in different laboratories.
Alice employs a collection of white mice to cluster a set $S$ of murine gene mutations, using
comparator $\{\prec_t\}_{t \in S}$ based on arithmetic on feature vectors. Meanwhile Bob uses his brown mice
 to cluster a set $S'$ of mutations, using
comparator $\{\prec'_x\}_{x \in S'}$. Although their sets of mice are disjoint, the
intersection $S \cap S'$ of mutations may be non-empty.

\noindent \textbf{Music streaming: } A streaming service specializing in Jazz merges with a service specializing in Pop.
Their recommender systems are similar enough to rank ``track to play next'' the same for tracks in the intersection of 
their catalogs, although the recommenders are constructed from different feature vectors.

\subsection{Glued comparators}\label{s:gluecompare}
Alice and Bob meet at a data science conference. They realize that there is a non-trivial intersection $S \cap S'$
of their genetic mutation sets. Although the feature vectors and similarity measures differ between their respective analyses,
they can at least verify a \textit{compatibility condition} between $\{\prec_t\}_{t \in S}$ and $\{\prec'_x\}_{x \in S'}$: if
$a, b, c$ are mutations in the intersection $S \cap S'$, then
Alice and Bob agree on the restrictions of their orientations to the boundary of the 2-simplex $\{ab, bc, ac\}$
when their respective comparators are applied to each mouse data set.
In other words, relations $a \prec_b c$ and $a \prec'_b c$ are equivalent
for $a, b, c \in S \cap S'$.
This condition is illustrated in the red arcs in Figure \ref{f:pushout}.

If the compatibility condition holds, Alice and Bob can glue their oriented simplicial complexes together, 
thanks to the sheaf theory developed in Section \ref{s:sheafview}.

\begin{figure}
\caption{
\textbf{TOP LEFT:} 
\textit{Concordant orientation of 1-simplices from objects $S:=\{a,b,c,t,u\}$.}
\textbf{TOP RIGHT:} 
\textit{Same for $S':=\{a,b,c,x,y\}$; rankings among $\{a, b,c\}$ same as for $S$.}
\textbf{BOTTOM LEFT:} 
\textit{Orientations above are glued together along 1-simplices 
$\{bc \rightarrow ab, ab \rightarrow ac, bc \rightarrow ac \}$.}
\textbf{BOTTOM RIGHT:} 
\textit{Twelve 1-simplices without an orientation. For example the enforcer triangle $\{ab, ax, at\}$ lacks orientation $ax \longrightarrow at$ 
or $ax \longleftarrow at$. Gluing succeeds despite lack of a ranking system on $S \cup S'$.
}
}
\label{f:pushout}
\begin{center}
\includegraphics[width=\linewidth]{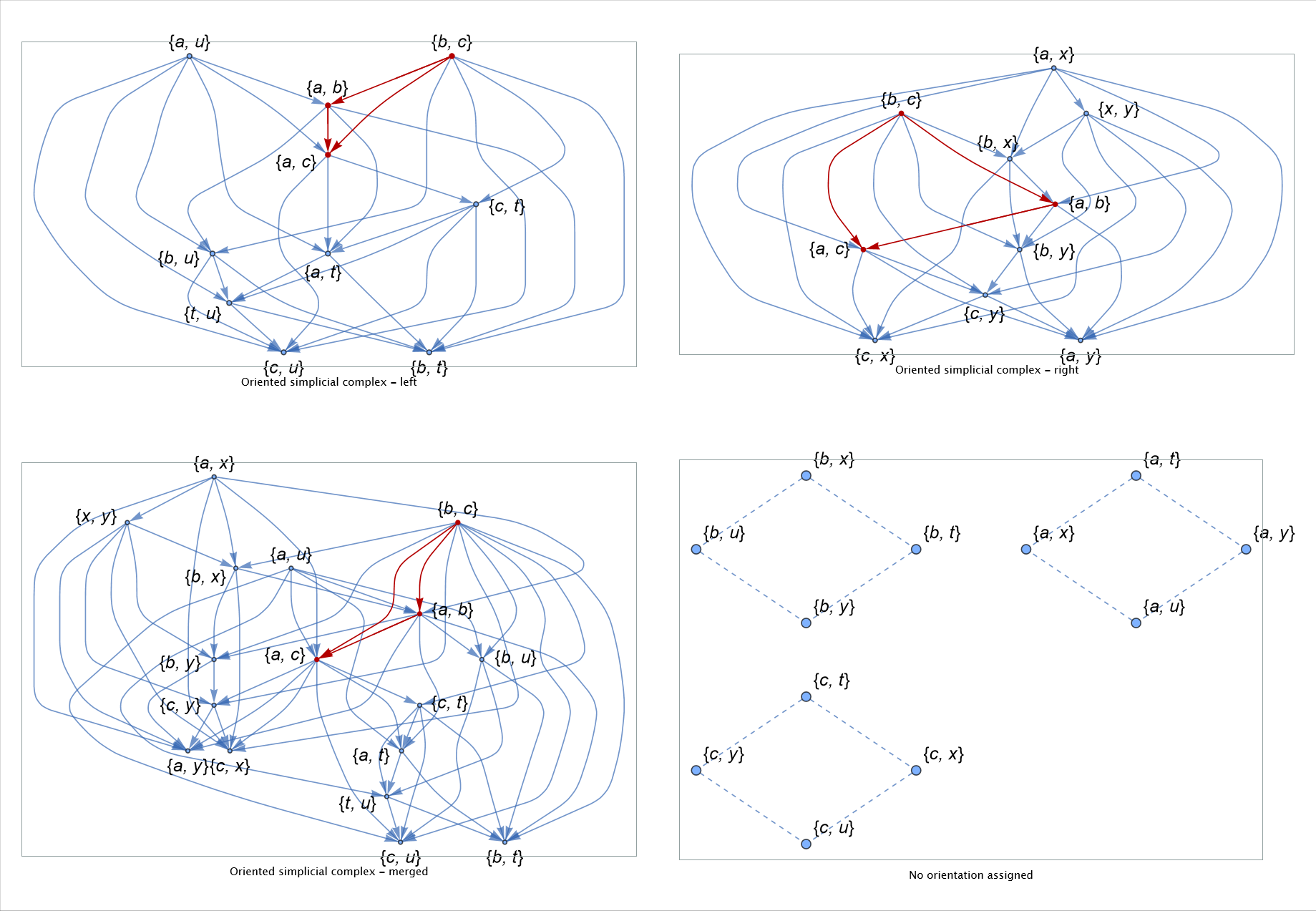} 
\end{center}
\end{figure}

\subsection{Sheaves can be glued, but comparators cannot} Under the compatibility condition, one might try to build 
a glued comparator $\{\prec''_z\}_{z \in S \cup S'}$ from comparators
$\{\prec_t\}_{t \in S}$ and $\{\prec_x\}_{x \in S'}$ thus:
$\prec''_z$ agrees with $\prec_z$ for $z \in S$, and agrees with $\prec'_z$ for $z \in S'$.
\textit{This fails in general} for reasons demonstrated in the last panel of Figure \ref{f:pushout}.
The problem arises when $a, b \in S \cap S'$, 
$x \in S' \setminus S$, and $t \in S \setminus S'$: the 1-simplex $\{ax, at\}$ lacks an orientation,
so enforcer triangle $\{ab, ax, at\}$ has an unoriented boundary edge, and thus $\prec''_a$ is not a total order 
on its friend set, but merely a partial order.

Here is a concrete example in the music streaming case. Say $b$ falls in both Jazz and Pop catalogs,
while $t$ is only in Pop and $x$ is only in Jazz. For the  comparator based at $b$, neither Jazz nor Pop supplies a mechanism to decide 
which of $t$ or $x$ is more similar to $b$.

As for voter triangles, let $t, u, v$ stand for elements of $S$, and let $x, y, z$ stand for elements of $S'$.
Figure \ref{f:pushout} makes it clear
that four types of voter triangle appear in $S \cup S'$:
 \begin{equation}\label{e:voter4types}
     \begin{tikzcd}{} & tu \arrow{dr}{} \arrow{dl}{} \\ uv \arrow{ur}{} \arrow{rr}{} &&
     \arrow{ul}{}  \arrow{ll}{} tv\end{tikzcd} 
     \begin{tikzcd}{} & tx \arrow{dr}{} \arrow{dl}{} \\ tu \arrow{ur}{} \arrow{rr}{} &&
     \arrow{ul}{}  \arrow{ll}{} ux\end{tikzcd} 
          \begin{tikzcd}{} & uy \arrow{dr}{} \arrow{dl}{} \\ xy \arrow{ur}{} \arrow{rr}{} &&
     \arrow{ul}{}  \arrow{ll}{} ux\end{tikzcd}  
          \begin{tikzcd}{} & yz \arrow{dr}{} \arrow{dl}{} \\ xz \arrow{ur}{} \arrow{rr}{} &&
     \arrow{ul}{}  \arrow{ll}{} xy\end{tikzcd}
     \end{equation}
     
The first and last of the voter triangles (\ref{e:voter4types}) can be understood entirely from the perspective
of $(S, \{\prec_t\}_{t \in S})$ and of $(S', \{\prec_x\}_{x \in S'})$, respectively.
However, the middle two in (\ref{e:voter4types}) involve interaction between the two ranking systems.
Some such voter triangles may be excluded because they fail
to be $\Gamma$-pertinent, for example if $\Gamma(x) \cap \{t, u\}$ is empty.

\subsection{Superior generality of the sheaf perspective}\label{s:sheafmerge}

The examples and the demonstration in Figure \ref{f:pushout} show why the sheaf perspective 
offers greater generality than the out-ordered digraph view when it comes to pooling ranking systems.
Our main sheaf result -- Corollary \ref{c:3sheaf} -- ensures that 3-concordant oriented simplicial complexes
can be glued together 
if orientations agree on their shared 1-simplices, and the result is 3-concordant oriented simplicial complex.
This orientation need not come necessarily from an
out-ordered digraph, because some enforcer triangles may be undefined. However the rank-base linkage algorithm 
may still be applied because partially out-ordered digraphs (Definition \ref{d:pood}) still exist at each vertex.

%%%%%%%%%%%%%%%%%%%%%%%%%%%%%%%%%%%%%%%%%%%%%%%%%%%%%%%
\section{Sampling from 3-concordant ranking systems}\label{s:sampling}

Nothing more will be said about the theory of rank-based linkage.
This section serves as a source of tools for generating
synthetic data for experiments with rank-based linkage.

\subsection{Why sampling is important}\label{s:whysample}
Machine learning algorithms applied to collections of points in
uncountable metric spaces rarely admit any kind of statistical null hypothesis,
because there is no uniform distribution on such infinite state spaces.

By contrast, the 3-concordant ranking systems on
$n$ points, for a specific $n$, form a finite set.
Hence uniform sampling from such a set may be possible. Indeed 
Table \ref{t:3conc10points}. presents a uniform random
3-concordant ranking table $T$, generated by rejection sampling.
The statistics of links arising from rank-based linkage applied to $T$ provide
a ``null model'', to which results of rank-based linkage on a specific data set may be
compared, in the same way that statisticians compare $t$-statistics for numerical
data to $t$-statistics under a null hypothesis.
With this in mind, we explore the limits of rejection sampling, and the
opportunity for Markov chain Monte Carlo.

\subsection{Ranking table}
Concise representation of a 3-concordant ranking system on a
set $S:=\{x_1, x_2, \ldots, x_n\}$ 
uses an $n \times n$ \textbf{ranking table} $T$, such as Table \ref{t:3conc10points}. 
Each row in table $T$
is a permutation of the integers $\{0, 1, 2, \ldots n-1\}$. 
The $(i, j)$ entry $r_i(j)$ is the rank awarded to object $x_j$ by the
ranking based at object $x_i$, taking $r_i(i)=0$. The advantage of this
representation is that it requires only $n^2$ entries, and makes all
enforcer triangles in the line graph acyclic automatically.

The condition for such a table to present a 3-concordant ranking system
is that, whenever $r_i(j) < r_i(k)$, then it is \textit{false} that
\begin{equation}\label{e:votercheck}
  r_j(k) < r_j(i) \quad \bigwedge \quad r_k(i) < r_k(j)  
\end{equation}
for otherwise the 2-simplex $\{ij, jk, ik\}$ would contain a 3-cycle, namely 
\[
  \begin{tikzcd}{} & ik \arrow{dr}{} \\ij \arrow{ur}{} \arrow[leftarrow]{rr}{} && jk\end{tikzcd} 
\]
Here we abuse notation, referring to the 1-simplex $x_i x_j$ as simply $i j$.

\subsection{Rejection sampling from ranking tables}
When $n \leq 10$, a uniformly random 3-concordant ranking system  can be found by
sampling uniformly from the set of ranking tables, and rejecting if any of the tests
(\ref{e:votercheck}) fails. Table \ref{t:rejectsample} shows some results, and
Table \ref{t:3conc10points} came from our largest experiment.

\begin{table}[]
    \centering
    \begin{tabular}{c|c|c|c|c|c}
Number of objects &  $n=6$  & $n=7$ & $n=8$ & $n=9$ & $n=10$ \\ \hline \hline
Number of attempts  & $10^6$ & $10^4$ & $10^3$ & 50 & 2 \\ \hline
Mean \# to success & 97.58  & 1,860.5 & 85,990 & $1.58 \times 10^7$ & $\frac{1}{2}(0.3 + 9.5) \times 10^9$\\ \hline
Rate of 4-cycles & 1.4\% & 1.3\% & 1.2\% & 1.2\% & 1.2\% or 0.7\% \\ \hline

    \end{tabular}
    \caption{
    \textit{Rejection sampling from ranking tables: }The reciprocal of the mean number of trials
    to success is an estimate of the proportion of 3-concordant ranking systems among
    all ranking systems. Statistics in the last row come from
   sampling $\binom{n}{4}$ 4-tuples (with replacement),
    and counting how many times the 4-tuple was a 4-loop, either in original or reversed order;
    this appears to be decreasing.
    The second $n=10$ trial took 16 hours, and the result appears
    in Table \ref{t:3conc10points}.}
    \label{t:rejectsample}
\end{table}

\subsection{Consecutive transposition in ranking tables}

Consider a graph $G_n = (R_n, F_n)$ whose vertex set $R_n$ consists of the 
ranking systems on an $n$-element set, each expressed in the form of a ranking table $T$.
Tables $T$ and $T'$ are adjacent under $F_n$
if $T'$ is obtained from $T$ by an operation called consecutive transposition, 
defined as follows. It is related to bubble sort and to Kendall tau distance \cite{kum}.

\begin{defn}\label{d:reflect}
Given an $n \times n$ ranking table $T$, a consecutive transposition means:
\begin{enumerate}
    \item Select a row $i$ for $1 \leq i \leq n$, and a rank $s$ between 1 and $n-2$.
    \item Identify elements $x_j$ and $x_k$ such that $r_i(j) = s$ and $r_i(k) = s+1$,
    i.e. $x_j$ and $x_k$ are ranked consecutively according to the ranking at $x_i$,
    so in particular $ij \rightarrow ik$.
     \item Transpose $x_j$ and $x_k$ in the ranking at $x_i$, meaning that
    we redefine $r_i(j) = s+1$ and $r_i(k) = s$, so in particular $ij \leftarrow ik$.
\end{enumerate}
\end{defn}

\begin{lem}\label{l:reflect}
If a ranking table $T$ presents a 3-concordant ranking system, then so does 
a ranking table $T'$ obtained from $T$ by a consecutive transposition, unless
\[
     r_j(i) < r_j(k) \, \bigwedge \, r_k(j) < r_k(i)
\] 
\end{lem}
\begin{proof}
Let $T'$ be obtained from $T$ as explained in Definition \ref{d:reflect}.
 The enforcer triangles of form $\{i k, i k', i k''\}$ are still
acyclic, because the rankings $r_i(\cdot)$ are unchanged except at 
ranks $s$ and $s+1$. For $\ell \neq i$, none of the rankings $r_{\ell}(\cdot)$
have changed so all enforcer triangles remain acyclic. The only
voter triangle which has changed is $\{ij, jk, ik\}$,  and the condition of the Lemma ensures that
the 2-simplex $\{ij, jk, ik\}$ does not contain a 3-cycle
in Step 4. Hence none of the triangles in $T'$ is a directed cycle.
\end{proof}

\subsection{Random walk on 3-concordant ranking systems}
We implemented a random walk on 3-concordant ranking systems using the following three steps:

\begin{enumerate}
    \item Generate a random permutation of the 0-simplices. An acyclic orientation of the 1-simplices follows
    by taking $xy \rightarrow xz$ whenever $xy$ precedes $xz$ in this random permutation. Initialize the
    ranking table using this orientation, which is concordant, and hence \textit{a fortiori} 3-concordant.
    \item 
    Generate a uniformly random consecutive transposition (Definition \ref{d:reflect}), and 
    perform the transposition unless the condition of Lemma \ref{l:reflect} applies. 
    \item 
    Repeat the last step many times. 
\end{enumerate}

\begin{defn}\label{d:rct-reflection}
    Random consecutive transposition with reflection at
    the 3-concordant boundary is the random walk defined above.
\end{defn}
By Lemma \ref{l:reflect}, the random walk of Definition \ref{d:rct-reflection} 
remains 3-concordant at every step.

It would be ideal to terminate when ``equilibrium'' is attained,
following Aldous and Diaconis \cite{ald} and
hundreds of later papers which cite \cite{ald}.
Unfortunately there is a difficulty. The Markov chain is not ergodic.
\begin{center}
\textit{The graph $G = (R_n, F_n)$ on the 3-concordant ranking systems 
need not be connected.}
\end{center}
Experiments revealed a counterexample when $n = 100$ of a vertex in $R_n$ which is 3-concordant,
without any 3-concordant neighbors in $G_n.$
Hence at best the random walk will settle to equilibrium in the component determined by the initial orientation. 

\begin{figure}
\caption{\textbf{Edge Flip Counterexample: }
\textit{
A 3-concordant ranking system on four points, expressed as an
orientation of its line graph (the octahedron). Both $\{v(1), v(2)\}$
and $\{v(3), v(4)\}$ are sources. Compare this to the orientation where both $\{v(1), v(2)\}$ and
$\{v(3), v(4)\}$ are sinks. We claim that, among the edges where these two orientations
differ, none of their directions can be reversed without creating a directed 3-cycle.
This shows the proof method of \cite{fuk} for constructing a path between two acyclic orientations
is no longer valid for 3-concordant orientations.
}
}
\label{f:counterex}
\begin{center}
\scalebox{0.4}{\includegraphics{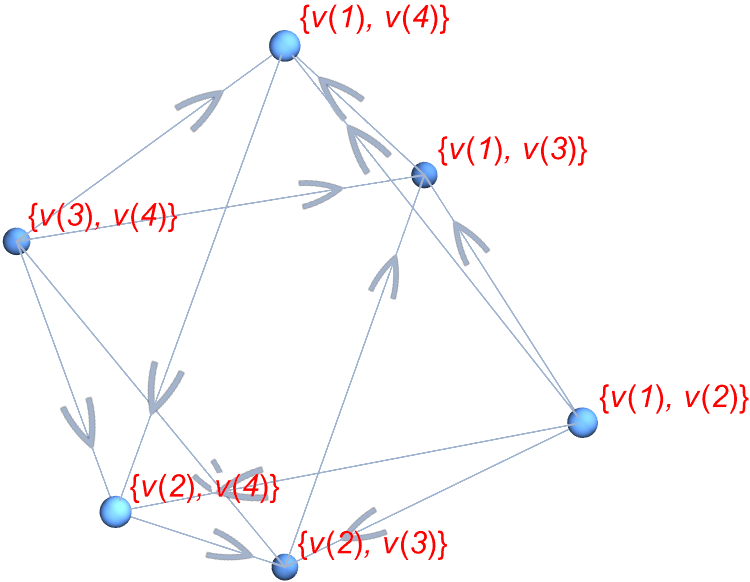} }
\end{center}
\end{figure}

By contrast, acyclic orientations form a connected graph, where adjacency
corresponds to a single edge flip.
Fukuda, Prodon, \& Sakuma \cite{fuk} prove
this by showing that, given two acyclic orientations $\omega, \omega'$,
there is always some edge where $\omega$ and $\omega'$ differ which can be flipped
to give another acyclic orientation.
This assertion no longer holds for 3-concordant orientations. A specific counterexample 
for $n=4$ is shown in Figure \ref{f:counterex}.

\subsection{Open problems}\label{s:open}
\subsubsection{The puzzle of the 4-cycles}
Exhaustive enumeration revealed exactly 450 3-concordant ranking systems on the 
four points $\{0, 1, 2, 3\}$. Of these, there are 8
in which the 4-loop $\{01, 12, 23, 04\}$ is a 4-cycle, in forward or reverse order.
To make Table \ref{t:3conc10points}, we constructed 
uniform random samples of 3-concordant ranking systems for $n = 6, 7, \ldots, 10$
In these samples, when a 4-tuple was sampled uniformly at random,
the proportion of 4-cycles (forwards or backwards) among
corresponding 4-loops was not $\frac{8}{450} = 1.777\cdots \%$.
Rather it seemed to decrease towards 1.2\%. It seems that the ranking system
induced on $\{0, 1, 2, 3\}$ by a uniform 3-concordant ranking system on 
$\{0, 1, 2, \ldots, n-1\}$ is not uniformly distributed among the 450 possibilities. 

We suspect that the explanation for this is as follows.  Let us choose a $3$-concordant
ranking system on $4$ points uniformly at random and consider extensions of it to a
fifth point.  It turns out that the expected number of possible extensions is
$e = 114248/75 = 1523.30667\dots$; on the other hand, if we consider only the $24$ systems that 
are not $4$-concordant, the expected number is only $e' = 4148/3 = 1382.666\dots$.  Accordingly
one might expect that if we choose a $3$-concordant ranking system uniformly on $n$ points,
the probability of its restriction to a given set of $4$ points not being $4$-concordant
would be about $24/450 \cdot (e'/e)^{n-4}$, so that the probability of a given $4$-loop
would be about $1/3$ of this.  This is roughly consistent with our observations; we do not
expect this model to be exactly correct, because the events are not independent.

\subsubsection{Intermediate sheaves}
One might wish to consider a condition on orientations that is stronger
than $3$-concordance, but that nevertheless determines a sheaf.  Is there a natural
way of associating a nonzero sheaf $\mcf_?(S)$ to a simplicial complex $S$ such that
$\mcf_?(S) \subseteq \mcf_3(S)$ for all $S$, functorially with 
respect to inclusions of simplicial complexes, where the inclusions are proper for
general $S$?%\footnote{This should probably be expressed in terms of the category of
%simplices with inclusions as morphisms, but the
%sheaves are on different simplices and it is no longer clear how to define a morphism.}
\subsubsection{Products of Cayley graphs}
Suppose we modify Definition \ref{d:rct-reflection} by skipping the check for 
the condition of Lemma \ref{l:reflect}. This unconstrained random walk is simply an $n$-fold product
of random walks on Cayley graphs; the Cayley graph for each factor is $\Gamma(\Sigma_{n-1}, \mathcal{S})$,
where $\Sigma_{n-1}$ is the symmetric group on $n-1$ elements, and the generating set $\mathcal{S}$ 
consists of permutations $(1, 2), (2, 3), \ldots, (n-2, n-1)$. For this Cayley graph, Bacher \cite{bac}
has shown that the eigenvalues of the Laplacian satisfy
\[
\lambda_2 = \ldots = \lambda_{n-1} = 2 - 2 \cos\left( \frac{\pi}{n-1} \right).
\]
This tells us about the rate of convergence to equilibrium of the unconstrained random walk.
The condition which forces reflection at the boundary of the 3-concordant ranking systems
has the effect of inducing dependence between the factors in this product of $n$ random walks
on copies of $\Gamma(\Sigma_{n-1}, \mathcal{S})$. Does convergence happen in about the same
time scale. Is it $O(n^3)$ steps, or $O(n^2 \log{n})$, or some other rate?

\subsubsection{Other Monte Carlo algorithm} Is there a better Markov chain
Monte Carlo algorithm for sampling
uniformly random 3-concordant ranking systems?
\\
\\
\noindent \textbf{Acknowledgments: } The authors thank Kenneth Berenhaut for introducing us to rank-based methods; the referee for valuable 
 improvements to the presentation, Figures \ref{f:nofriendsofx} \& \ref{f:noglue3} in particular; 
 Marcus Bishop for help with Figure \ref{f:k4};
 and the Tutte Institute (Ottawa) for hosting a workshop
 where this project began.

%%%%%%%%%%%%%%%%%%%%%%%%%%%%%%%%%%%%%%%%%%%%%%%%%%%%%%%%%%%%

\end{document}